\documentclass[a4paper,11pt]{amsart}

\usepackage[utf8x]{inputenc}
\usepackage{amsmath}
\usepackage{amsfonts}
\usepackage{amssymb}
\usepackage{mathrsfs}                  
\usepackage{amsthm}

\usepackage{amscd}
\usepackage{xcolor}
\usepackage{a4wide}
\usepackage[hyperfootnotes=false]{hyperref}
\usepackage{bbm}
\usepackage{mathtools}
\usepackage{geometry}
\usepackage{graphicx}
\usepackage{epstopdf}
\numberwithin{equation}{section}

\theoremstyle{plain}
\newtheorem{theorem}{Theorem}[section]
\newtheorem{proposition}[theorem]{Proposition}
\newtheorem{lemma}[theorem]{Lemma}

\theoremstyle{definition}

\theoremstyle{remark}
\newtheorem{remark}[theorem]{Remark}

\newtheoremstyle{com}{}{}{\color{blue}}{}{\color{blue}}{}{ }{}
\theoremstyle{com}

\newcommand{\dd}{\,\mathrm{d}}

\newcommand{\E}{\mathbb{E}}

\newcommand{\R}{\mathbb{R}}
\newcommand{\C}{\mathbb{C}}
\newcommand{\N}{\mathbb{N}}
\newcommand{\tr}{\mathrm{tr}}

\renewcommand{\d}{\mathrm{d}}
\renewcommand{\P}{\mathbb{P}}
\newcommand{\Q}{\mathbb{Q}}
\newcommand{\Pm}{\mathcal{P}}

\renewcommand{\Re}{\mathrm{Re}\,}
\newcommand{\U}{\mathcal{U}}

\newcommand{\A}{\mathcal{A}}
\newcommand{\F}{\mathbb{F}}
\newcommand{\M}{\mathcal{M}^+}

\title{Linearized Filtering of Affine Processes Using Stochastic Riccati Equations}

\author[Lukas Gonon]{Lukas Gonon}
\address{Lukas Gonon, Eidgen\"ossische Technische Hochschule Z\"urich, Switzerland}
\email{lukas.gonon@math.ethz.ch}
\author[Josef Teichmann]{Josef Teichmann}
\address{Josef Teichmann, Eidgen\"ossische Technische Hochschule Z\"urich, Switzerland}
\email{josef.teichmann@math.ethz.ch}

\begin{document}


\begin{abstract}
  We consider an affine process $ X $ which is only observed up to an
  additive white noise, and we ask for the law of $ X_t $, for some
  $ t > 0 $, conditional on all observations up to time $ t $. This is
  a general, possibly high dimensional filtering problem which is not
  even locally approximately Gaussian, whence essentially only
  particle filtering methods remain as solution techniques. In this
  work we present an efficient numerical solution by introducing an
  approximate filter for which conditional characteristic functions
  can be calculated by solving a system of generalized Riccati
  differential equations depending on the observation and the process
  characteristics of $X$. The quality of the approximation can be
  controlled by easily observable quantities in terms of a macro
  location of the signal in state space. Asymptotic techniques as well
  as maximization techniques can be directly applied to the solutions
  of the Riccati equations leading to novel very tractable filtering
  formulas. The efficiency of the method is illustrated with numerical
  experiments for Cox–Ingersoll–Ross and Wishart processes, for which
  Gaussian approximations usually fail.
\end{abstract}
\thanks{Both authors thank for the generous support of ETH Foundation and of SNF}

\maketitle
\frenchspacing

\noindent\textbf{Key words and phrases: affine process, filtering, conditional law, Riccati equation, Zakai equation.}   \\
\textbf{MSC 2010 Classification: 60G35, 62M20}

\section{Introduction}
Consider a time-dependent stream of multi-variate signals which can
not be observed directly, but only through a stream of noisy
measurements. Given the observations made up to a specified moment in
time $t$, what can we optimally say about the signal at time $t$,
i.e. what is the best estimate for the signal's value? There are
various mathematical formulations of this fundamental
problem. Research fields such as time-series analysis, signal
processing and (frequentist) non-parametric statistics model the
signal process as a deterministic function or focus on a discrete-time
setting. In \textit{stochastic filtering} the signal and observation
processes are modeled as continuous-time stochastic processes, i.e.~a
(dynamic) Bayesian perspective is adopted.

The mathematical formulation of the stochastic filtering problem is
the following: consider a $D$-valued stochastic process $X$, a
$p$-dimensional Brownian motion $W$ and an observation function
$h \colon D \to \R^p$. Define the observation process $Y$ as
\begin{equation}\label{eq:Yintro} Y_t = \int_0^t h(X_s)\dd s + W_t,
  \quad t \geq 0. \end{equation}
Note that both $X$ and $W$ are defined on \emph{one} probability space $(\Omega,\mathcal{F},\P)$, $p \in \N$, the state space $D$ is some set and some regularity on $h$ and the sample paths of $X$ needs to be imposed to make \eqref{eq:Yintro} well-defined. Here $X$ models the signal and $Y$ the observation process. The \textit{filtering problem} is to calculate $\pi_t$, the conditional distribution of $X_t$ given $\mathcal{F}_t^Y := \sigma(Y_s \colon s \in [0,t])$, i.e.~the observations up to time $t$, for each $t \geq 0$.

Starting in the mid-twentieth century, stochastic filtering has
received an enormous amount of attention and has influenced many
fields of mathematics -- we refer to the introductory textbooks
\cite{shiryaev2001}, \cite{Bain2009}, the historical overview in
\cite{crisan2014} and the handbook \cite{crisan2011}. Theoretically
the filtering problem has been solved: $(\pi_t)_{t \geq 0}$ can be
characterized as the unique solution to a measure-valued stochastic
differential equation (the Fujisaki-Kallianpur-Kunita or
Kushner-Stratonovich equation). For applications, e.g., in
mathematical finance \cite{Brigo1998} or geophysics \cite{law2015},
also a quick numerical calculation of $\pi_t$ is quintessential -- in
fact for any application of a continuous-time stochastic model that
features latent factors. It has been shown that apart from a few
special cases, e.g.~when $h$ is affine and $X$ is an
Ornstein-Uhlenbeck process or when the state space $D$ consists of
finitely many points, the equation for $(\pi_t)_{t \geq 0}$ is truly
infinite-dimensional. As a consequence, devising numerical methods to
calculate $\pi_t$ or even just the conditional mean
$\E[X_t|\mathcal{F}^Y_t]$ is challenging. In most cases it is
in-feasible due to computational constraints. Standard numerical
methods (\cite[Chapters 8-10]{Bain2009}) either only work for
low-dimensional state spaces or for approximately Gaussian
setups.\footnote{In high-dimensional geophysical applications for
  example, only approximate Gaussian filters are routinely used (see
  the preface of \cite{law2015}).}  However, post-crisis financial
modeling asks for factor processes $X$ which are both high-dimensional
and not approximately Gaussian. The lack of numerical filtering
methods for such processes has put serious limitations on the modeling
flexibility: one has not been able to include latent factors in them.

In the present article, we fill this gap and show that the narrow
class of processes for which an efficient numerical solution is
possible (see above) also includes \textit{affine processes}.  More
precisely, we consider the case when $h$ is affine and the signal
process $X$ is an \textit{affine process} with state space
$D=\R_+^m \times \R^{d-m}$ as characterized in \cite{duffie2003}. This
class of processes includes for example L\'evy processes,
Cox-Ingersoll-Ross processes \cite{Cox1985} or the Heston model
\cite{Heston1993} and is very widely used in financial applications
(see e.g.\ \cite{duffie2003}, \cite[Section~3]{keller-ressel2015} for
a list of references). The filtering problem arises naturally in this
context; for example, $X$ could model the short rate and $Y$ the
observed yields of bond prices as in \cite{Geyer1999},
\cite{Chen2003}, see also \cite{Brigo1998}.

Let us briefly summarize the key ideas of our approach. As a first
step the distribution of $X$ conditional on $ \mathcal{F}^Y$ is
rewritten in terms of the pathwise filtering functional as studied by
\cite{Davis1980}, \cite{Clark1978}. Although the functional itself is
not directly tractable, it can be approximated by a linearized version
thereof. This new \textit{linearized filtering functional} (LFF) is
numerically tractable, since the Fourier coefficients can be
calculated by solving a system of generalized Riccati equations with
vector fields depending on the observation $Y$.  This gives rise to
Fourier filtering techniques, analogously to the Fourier pricing
techniques used for affine (log-price) models, see e.g.\
\cite{Carr1999} and \cite{duffie2003}. In addition the (approximate)
conditional moments can be calculated by solving a system of ordinary
differential equations. In contrast to existing numerical methods
(e.g. a particle filter), this is very well-suited to parallel
computations and thus promising for high-dimensional filtering.

There is also another equally fruitful viewpoint on this approach:
the Zakai equation for the (un-normalized) distribution $ \sigma_t $
is a stochastic partial differential equation (SPDE) of the following
form (under mild regularity conditions)
\[
  d \sigma_t (dx) = \mathcal{A}^* \sigma_t (dx) dt + \sum_{i=1}^d
  h_i(x) \sigma_t(dx) dY^i_t \, ,
\]
where $ \mathcal{A}^* $ denotes the adjoint of the generator of
$X$. This equation, even though linear, has a quite complicated
geometry: essentially only the Kalman filter, which corresponds to an
Ornstein Uhlenbeck process $X$ and linear observation $ h(x) = x $
allows for a finite dimensional realization, i.e.~a way to write the
SPDE's solution via solutions of finite dimensional stochastic
differential equations. This is due to the fact that the geometrically
relevant Stratonovich formulation of the equation has an additional
term of type $ h(x)^2 \sigma_t (dx) $ in the drift, which in turn
causes the infinite dimensional analogon of hypo-ellipticity, whence
no finite dimesional realizations can exist. The only way to cure this
phenomenon in the relevant Brownian case is by replacing the
Stratonovich correction by a linear expression, which is of course
locally possible in a well controlled way. Then this \emph{modified
  Zakai equation} has a completely different solution structure which
can often be described by finite dimensional stochastic differential
equations. In case of general affine processes $X$, even beyond the
canonical setting used in this article, the modified Zakai equation
under linear observation can be considered as an \emph{affine} SPDE
with time-dependent affine potential term (for this interpretation one
necessarily needs the Stratonovich formulation), whose solution can be
described by generalized stochastic Riccati equations. Notice also
that the modification of the Zakai equation depends on the nature of
the Stratonvich correction, in particular in case of finite variation
noises the modification would vanish and we would actually have a
solution theory for the classical Zakai equation by Fourier methods.

All of this is explained in detail in
Section~\ref{sec:linearizedFiltering}, while
Section~\ref{sec:background} provides background on affine processes
and the filtering problem. The proofs of the statements on the LFF as
well as local existence and uniqueness of solutions to the Riccati
equations in Section~\ref{sec:linearizedFiltering} are then given in
Section~\ref{sec:proofsAff}. They are based on a change of measure and
comparison results for generalized Riccati equations. These are of
independent interest and extend results from \cite{Kallsen2010} and
\cite{keller-ressel2015} to Riccati equations associated to
non-conservative time-inhomogeneous affine ``processes'' that do not
necessarily satisfy the admissibility conditions.
 
This theoretical analysis is complemented by a numerical study. In
Section~\ref{sec:CIR} the methodology is applied to the problem of
filtering a Cox-Ingersoll-Ross (CIR) process. In numerical examples
the filter induced by the linearized filtering functional, the
\textit{affine functional filter} (AFF), is compared to the benchmark
(a bootstrap particle filter) and two standard Gaussian- and
Gamma-approximation approaches (extended Kalman filter and
\cite{Bates2006}). Not only is the AFF very close to the benchmark
(and in particular more accurate than the two approximations), but it
can also be calculated more efficiently than a particle filter. In
examples in higher dimensions the situation turns out to be even more
extreme: In Section~\ref{sec:Wishart} the methodology is applied to
Wishart processes \cite{Bru1991}, a matrix-valued extension of CIR
processes (and a special case of affine processes taking values in
$S_d^+$, the set of symmetric positive semi-definite matrices). While
in theory particle methods are applicable to this problem, in practice
this requires enormous computational resources. Numerical experiments
(already) for $d=3$ show that in order to achieve the same level of
accuracy (measured in terms of mean square error) as the AFF an
outrageous number of particles would be necessary.  Conversely, if one
only uses a number of particles yielding similar computing times for
the two methods, the mean-square error of a bootstrap particle filter
is still by far larger than the error of the AFF.  This makes the AFF
the first numerically feasible method for filtering Wishart processes.
 
\subsection{Notation}

Fix a complete probability space $(\Omega,\mathcal{F},\P)$ on which
all random variables are defined.

Fix $p \in \mathbb{N}$, $m \in \mathbb{N} \cup \{0\}$,
$d \in \mathbb{N}$ with $d \geq m$ and set
$D = \R_+^m \times \mathbb{R}^{d-m}$. Let
$\langle \cdot , \cdot \rangle$ denote the standard inner product on
$\R^d$ and $|\cdot|$ the associated norm. Also write
$\langle \cdot , \cdot \rangle$ for the linear extension of the inner
product to $\R^{d} + i \R^{d}$, but without complex conjugation.  Set
\[ I:= \{1,\ldots,m \}, \quad J:= \{m+1,\ldots,d \}.  \] For
$k \in \N$, write
\[ \C_{-}^k = \{ u \in \C^k \,:\, \Re u_i \leq 0, \, \forall i \},
  \quad \C_{--}^k = \{ u \in \C^k \,:\, \Re u_i < 0, \, \forall i
  \} \] and define $\U = \C_{-}^m \times i \R^n$ .

Denote by $B(D)$ and $C_b(D)$ the sets of bounded measurable functions
and bounded continuous functions on $D$ and by $\Pm(D)$ the set of
probability measures on $D$. As usually, $\Pm(D)$ is equipped with the
topology of weak convergence. Let $\M(D)$ denote the set of finite
measures on the Borel $\sigma$-algebra $\mathcal{B}(D)$. Given
$\mu \in \M(D)$ and a measurable, $\mu$-integrable function $f$ on
$D$, write $\mu f := \int_D f(x) \mu(\d x)$.

Fix a continuous truncation function $\chi \colon \R^d \to [-1,1]^d$
with $\chi(\xi) = \xi$ in a neighborhood of $0$ and bounded away from
$0$ outside that neighborhood. In fact, in order to be able to rely on
a result from \cite{Kallsen2010} for $k=1,\ldots d$ we choose
\[ \chi_k(x)=\begin{cases} 0 \quad & \text{ if } x_k = 0, \\ (1 \wedge
    |x_k|)\frac{x_k}{|x_k|} \quad & \text{ otherwise.}
  \end{cases} \]

Let $\pi_0 \in \Pm(D)$, $D(\mathcal{L}) \subset C_b(D)$ and
$\mathcal{L}\colon D(\mathcal{L}) \to C_b(D)$ linear. Recall that a
$D$-valued stochastic process $(X_t)_{t \geq 0}$ defined on some
probability space
$(\tilde{\Omega},\tilde{\mathcal{F}},\tilde{\mathbb{P}})$ is called a
\textit{solution to the martingale problem for
  $(D(\mathcal{L}),\mathcal{L},\pi_0)$}, if
$\tilde{\mathbb{P}} \circ X_0^{-1} = \pi_0$ and for each
$h \in D(\mathcal{L})$, the process
\[
  h(X_t) - h(X_0) - \int_0^t \mathcal{L} h (X_s) \dd s, \quad t \geq
  0,
\]
is a martingale (in its own filtration). The martingale problem for
$(D(\mathcal{L}),\mathcal{L},\pi_0)$ is said to be \textit{well-posed}
if there exists a solution and any two solutions have the same
finite-dimensional marginal distributions.

\section{Background: Affine processes and the filtering
  problem}\label{sec:background}

\subsection{Affine processes}
\subsubsection{Definition and characterization} \label{sec:affine} Let
us review the definition of an affine process and some consequences
thereof. We refer to \cite{duffie2003}, \cite{Keller-Ressel2011} and
\cite{Cuchiero2013} for further details and references.

Consider a $D$-valued time-homogeneous Markov process
$((X_t)_{t \geq 0},(\P_x)_{x \in D})$ defined on
$(\Omega,\mathcal{F})$, see \cite[Chapter III]{Rogers2000}. Denote by
$(P_t)_{t \geq 0}$ the associated semigroup on $B(D)$ and assume
$P_t 1 = 1$ for all $t \geq 0$ (i.e. the process is
conservative). $((X_t)_{t \geq 0},(\P_x)_{x \in D})$ is called affine,
if it is stochastically continuous, $X$ has RCLL-paths
($\P_x$-a.s. for any $x \in D$) and there exist functions
$\phi:\R_{\geq 0} \times \U \to \C$ and
$\psi: \R_{\geq 0} \times \U \to \C^d$ such that for all $x \in D$,
$(t,u) \in \R_{\geq 0} \times \U$:
\begin{equation}\label{eq:affineDef}
  \E_x[e^{\langle X_t , u\rangle}] = \exp(\phi(t,u) + \langle x, \psi(t,u) \rangle).
\end{equation}

\begin{remark} As shown in \cite{Keller-Ressel2011} this definition
  implies that for all $u \in \U$,
  \begin{equation}\label{eq:affineRegular}
    F(u):= \left. \frac{\partial \phi}{\partial t}(t,u)\right\rvert_{t = 0+} \, , \quad R(u):= \left. \frac{\partial \psi}{\partial t}(t,u) \right\rvert_{t = 0+} 
  \end{equation}
  exist and are continuous at $u = 0$. Thus, in the terminology of
  \cite{duffie2003} we are considering a conservative, regular affine
  process.

\end{remark}

\begin{remark}
  Alternatively, we could only assume that
  $((X_t)_{t \geq 0},(\P_x)_{x \in D})$ is conservative,
  stoch\-astically continuous and \eqref{eq:affineDef} holds for
  $(t,u) \in \R_{\geq 0} \times i\R^d$. Then \cite{Keller-Ressel2011}
  implies that it is a Feller process and in particular, we may choose
  an RCLL version of $X$ on $D$ (under $\P_x$, for any
  $x \in D$).\footnote{Since the process is conservative, there is no
    need to consider the one-point compactification of $D$.} Finally
  \cite[Theorem 2.7]{duffie2003} implies that \eqref{eq:affineDef} can
  be extended to $\R_{\geq 0} \times \U$.
\end{remark}

Let us now review some key properties of affine processes. To
formulate these, an additional definition is required: A collection of
parameters
\begin{equation}\label{eq:collectionOfParameters}
  (a,\alpha,b,\beta,c,\gamma,\mu^0,\mu) \end{equation}
is called admissible, if it satisfies the following (admissibility) conditions:
\begin{align}\label{eq:admiss1}
  a \in \mathrm{Sem}^d \text{ with } a_{i,j} = 0  \, & \text{ for } \quad i,j \in I \\ \label{eq:alphacond}
  \alpha=(\alpha^1,\ldots,\alpha^m) \text{ with } \alpha^i \in \mathrm{Sem}^d \text{ and } \alpha_{k,j}^i= 0 & \text{ for } \quad k,j \in I \setminus \{i\} \\ \label{eq:driftcond1}
  b \in \R^d \text{ with } b_i - \int_{D \setminus \{0\}} \chi_i(\xi) \mu^0(\d \xi) \geq  0  \, & \text{ for } \quad i \in I \\ \label{eq:driftcond2}
  \beta \in \R^{d \times d} \text{ with } \beta_{i,j} - \int_{D \setminus \{0\}} \chi_i(\xi) \mu^j(\d \xi) \geq  0  \, & \text{ for } \quad i,j \in I \text{ and } i \neq j\\ \label{eq:driftcond3}
  \beta_{i,k}=0  \, & \text{ for } \quad i\in I, \, k \in J
  \\ \label{eq:ccond}
  c \in \R_+  &
  \\ \label{eq:gamcond}
  \gamma \in \R^m_+  & 
  \\ \mu=(\mu^1,\ldots,\mu^m) \text{ and for } i \in I \cup \{0\}, 
  \mu^i &\text{ is a Borel measure on } D\setminus \{0\}
  \\ \label{eq:admissLevyMes1} \int_{D\setminus\{0\}} \chi_k(\xi) \mu^i(\d \xi) < \infty \,  \text{ for } \quad  & i \in I \cup \{0\}, \, k\in I \setminus \{i\}
  \\ \label{eq:admissLevyMes2} \int_{D\setminus\{0\}} \chi_k(\xi)^2 \mu^i(\d \xi) < \infty \,  \text{ for } \quad & i \in I \cup \{0\}, \, k \in (J \cup \{i\})\setminus\{0\}
\end{align}

\begin{remark}\label{rmk:admissibility}
  The admissibility conditions are identical with
  \cite[Definition~2.6]{duffie2003}. We have only changed notation
  slightly in order to match the semimartingale notation in
  \cite{Kallsen2010}. The measure $m$ in
  \cite[Definition~2.6]{duffie2003} is denoted $\mu^0$ here, the
  truncation function is arbitrary (as in \cite{Filipovic2005}) and we
  denote by $b, \beta$ the parameters $\tilde{b},\tilde{\beta}$ from
  \cite[Theorem~2.12]{duffie2003}. Our conditions
  \eqref{eq:driftcond1}, \eqref{eq:driftcond2} for these are
  equivalent to conditions (2.6) and (2.7) in
  \cite[Definition~2.6]{duffie2003}. This leads to different
  expressions below for \eqref{eq:FandR} and
  \eqref{eq:affineGeneratorShort} than in \cite{duffie2003}, see also
  \cite[Remark~2.13]{duffie2003}.
\end{remark}

Suppose $((X_t)_{t \geq 0},(\P_x)_{x \in D})$ is an affine process and
denote again by $(P_t)_{t \geq 0}$ the restriction of the associated
semigroup to $C_0(D)$. Then (see \cite[Theorem~2.7, Theorem~2.12 and
Proposition~9.1]{duffie2003}) there exists a collection of admissible
parameters\footnote{Recall that we only consider conservative affine
  processes here.} \eqref{eq:collectionOfParameters} with $c=0$ and
$\gamma=0$ such that the following properties hold:
\begin{itemize}
\item $F$ and $R$ in \eqref{eq:affineRegular} are given as
  \begin{equation}\label{eq:FandR} \begin{aligned}
      F(u)&=  \frac{1}{2}\langle u, a u \rangle + \langle b , u \rangle + \int_{D \setminus \{0\}}\left(e^{\langle \xi, u \rangle } - 1 - \langle \chi(\xi), u \rangle \right) \mu^0(\d \xi) \\
      R_i(u) & = \frac{1}{2}\langle u, \alpha^i u \rangle + \langle
      \beta^i , u \rangle + \int_{D \setminus \{0\}}\left(e^{\langle
          \xi, u \rangle } - 1 - \langle \chi(\xi), u \rangle \right)
      \mu^i(\d \xi)
    \end{aligned}
  \end{equation}
  for $i = 1,\ldots, m$ and $R_i(u) = \langle \beta^i , u \rangle$ for
  $i = m+1,\ldots,d$. Here $\beta^i \in \R^d$ is defined via
  \[ \beta^i_j:=\beta_{j,i}, \quad \text{ for } \quad 1 \leq i,j \leq
    d.\]
\item $\phi$ and $\psi$ solve the generalized Riccati equations
  \begin{equation}\label{eq:genRiccati}\begin{aligned}
      \partial_t \phi(t,u) & = F(\psi(t,u)), \quad \phi(0,u) = 0 \\
      \partial_t \psi(t,u) & = R(\psi(t,u)) , \quad \psi(0,u) = u
    \end{aligned}
  \end{equation}
  for $t \geq 0, u \in \U$.

\item $(P_t)_{t \geq 0}$ is a Feller semigroup (in the sense of
  \cite[Chapter~III]{Revuz1999}). Denote by $(D(\A),\A)$ its
  infinitesimal generator. Then $C_c^\infty(D)$ is a core for $\A$,
  $C_0^2(D) \subset D(\A)$ and for any $f \in C_0^2(D)$, $x \in D$,
  \begin{equation}\label{eq:affineGeneratorShort}\begin{aligned}
      \A f (x) & = \frac{1}{2} \sum_{k,l=1}^d \alpha_{k l}(x)
      \frac{\partial^2 f(x)}{\partial x_k \partial x_l} + \langle
      \beta(x) , \nabla f(x) \rangle \\ & + \int_{D \setminus \lbrace
        0\rbrace} (f(x+\xi)-f(x)- \langle \chi(\xi), \nabla f(x)
      \rangle ) K(x,\d \xi)
    \end{aligned}\end{equation}
  where 
  \begin{equation}\label{eq:affineGeneratorShortCoeff}\begin{aligned}
      \alpha (x) & = a + \sum_{i=1}^m \alpha^i x_i \\
      \beta (x) & =  b + \sum_{i=1}^d \beta^i x_i \\
      K(x,\d \xi) & = \mu^0(\d \xi) +\sum_{i=1}^m x_i \mu^i(\d \xi).
    \end{aligned}\end{equation}

\item $X$ is a semimartingale (under $\P_x$, for any $x \in D$)
  admitting characteristics $(B,C,\nu)$ with respect to $\chi$ given
  by
  \begin{equation}\label{eq:characteristics}
    B_t = \int_0^t \beta(X_s) \dd s, \quad C_t = \int_0^t \alpha(X_s) \dd s , \quad \nu(\d t,\d\xi) = K(X_t,\d \xi)\dd t,
  \end{equation}
  where $\alpha, \beta, K$ are as in
  \eqref{eq:affineGeneratorShortCoeff}.
\end{itemize}

Finally, let us put (conservative) affine processes into the framework
of \cite{Ethier1986a}. This is the purpose of
Lemma~\ref{lem:martingaleProblem} below. It is very close to
\cite[Lemma~10.2]{duffie2003}, but considers arbitrary initial laws
and establishes uniqueness also within the class of solutions to the
martingale problem which are not necessarily RCLL. This extension is
required to establish uniqueness for evolution equations (as the Zakai
equation in Theorem~\ref{thm:zakaiUniqueness} below) associated to
$\A$.

\begin{lemma}\label{lem:martingaleProblem} Fix a collection of
  admissible parameters \eqref{eq:collectionOfParameters} and define
  $\A_0$ as the restriction of $\A$ (see
  \eqref{eq:affineGeneratorShort}) to $C_c^\infty(D)$. Then for any
  $\pi_0 \in \Pm(D)$, the martingale problem for
  $(C_c^\infty(D),\A_0, \pi_0)$ is well-posed and the solution has
  RCLL-sample paths.
\end{lemma}
\begin{proof}
  The statement of \cite[Theorem 2.7]{duffie2003} that $X$ is a Feller
  process means that $(P_t)_{t \geq 0}$ is a strongly continuous,
  positive contraction semigroup on $C_0(D)$ in the terminology of
  \cite{Ethier1986a}.  Furthermore, by \cite[Chap.4,
  Cor.~2.8]{Ethier1986a} and since $X$ is conservative,
  $(D(\A),\mathcal{A})$ is conservative (in the terminology of
  \cite{Ethier1986a}). Thus $(P_t)_{t \geq 0}$ is a Feller semigroup
  (on $C_0(D)$) also in the terminology of \cite{Ethier1986a}. Set
  $D(\A_0)=C_c^\infty(D)$.  By \cite[Theorem 2.7]{duffie2003},
  $D(\A_0)$ is a core for $(D(\A),\A)$ and so the closure of the
  operator $(D(\A_0),\A_0)$ is again $(D(\A),\A)$. Combining
  \cite[Chap.4, Thm. 2.2, 2.7 and 4.1]{Ethier1986a} then yields the
  statement.
\end{proof}

In view of Lemma~\ref{lem:martingaleProblem} the following terminology
is sensible: Fix $\pi_0 \in \Pm(D)$. We call an RCLL stochastic
process $X$ on $(\Omega,\mathcal{F},\P)$ an affine process started
from $\pi_0$, if it is a solution to the martingale problem for
$(C_c^\infty(D),\A_0, \pi_0)$. By Lemma~\ref{lem:martingaleProblem}
this uniquely determines the law of $X$ under $\P$.

\subsubsection{Exponential moments of affine processes}
For the analysis of this article, it will be necessary to extend
\eqref{eq:affineDef} to $U_0 \subset \R_{\geq 0} \times \C^d$, where
$U_0$ is open and $0 \in U_0$. This means that an assumption on
exponential moments is required. Suppose that
\begin{equation}\label{eq:expMoments}
  \int_{D \setminus \{|z| \leq 1\}} |z| e^{\langle z , u\rangle} \mu^i(\d z) < \infty \quad \text{ for all }i=0,\ldots, m \text{ and } u \in \R^d.
\end{equation} 
Suppose $((X_t)_{t \geq 0},(\P_x)_{x \in D})$ is an affine process and
define
\begin{equation}
  E = \{ (t,u) \in \R_{\geq 0} \times \R^d \; : \; \E_x[e^{\langle X_t , u \rangle}] < \infty \, \text{ for all } \; x \in D \}.
\end{equation}
By definition, this is the maximal domain on which the left hand side
of \eqref{eq:affineDef} is finite. Under assumption
\eqref{eq:expMoments}, $E$ is open, $0 \in E$ and $\phi$ and $\psi$
can be extended to $E$. This is summarized in the next Lemma, which
directly follows from \cite{keller-ressel2015} and
\cite{filipovic09}. See also \cite{spreij2010} and further references
in all these articles.

\begin{lemma} Suppose \eqref{eq:expMoments} holds. Then
  \begin{itemize}
  \item[(i)] $E$ is open in $\R_{\geq 0} \times \R^d$,
  \item[(ii)] for any $(T,u) \in \R_{\geq 0} \times \C^d$ with
    $(T,\Re u) \in E$, there exists a unique solution to
    \eqref{eq:genRiccati} on $[0,T]$ and \eqref{eq:affineDef} holds.
  \end{itemize}
\end{lemma}
\begin{proof}
  By \cite[Lemma~5.3]{duffie2003} and \eqref{eq:expMoments}, $F$ and
  $R$ are analytic functions. Therefore the same reasoning as in the
  proof of \cite[Lemma~2.3]{filipovic09} shows that for any
  $u \in \C^d$, there exists $t_+(u) \in (0,\infty]$ such that
  \eqref{eq:genRiccati} has a unique solution on $[0,t_+(u))$ and the
  set
  \[ D_\R := \lbrace (t,y) \in \R_{\geq 0} \times \R^d \; : \; t <
    t_+(y) \rbrace \] is open in $\R_{\geq 0} \times
  \R^d$. Furthermore, by \cite[Theorem~2.14(b)]{keller-ressel2015},
  \cite[Theorem~2.17(b)]{keller-ressel2015} and \eqref{eq:expMoments},
  one has $D_\R \subset E$ and \eqref{eq:affineDef} holds for all
  $(t,u) \in D_\R$, $x \in
  D$. \cite[Theorem~2.14(a)]{keller-ressel2015} implies
  $E \subset D_\R$ and hence $E = D_\R$. This shows (i).
  $(T,\Re u) \in E$ yields $(T,\Re u) \in D_\R$ and so
  \cite[Theorem~2.26]{keller-ressel2015} implies (ii).
\end{proof}

A further consequence of \eqref{eq:expMoments} is the following:

\begin{lemma}\label{lem:affMoments} Assume \eqref{eq:expMoments}. Then
  for any $T \geq 0$, $k \in \mathbb{N}$, $x \in D$
  \begin{align}\label{eq:momentsPolyn} \E_x[|X_T|^{2 k}] & < \infty, \\
    \label{eq:momentsPolynIntegral} \E_x\left[\int_0^T|X_t|^{2 k} \dd t\right] & < \infty.
  \end{align}
\end{lemma}
\begin{proof}
  By \cite[Lemma 5.3]{duffie2003} and \eqref{eq:expMoments}, $F$ and
  $R$ are analytic functions on $\C^d$. Thus by
  \cite[Lemma~6.5(i)]{duffie2003}, $\phi$ and $\psi$ are in
  $C^{\infty}(\R_+ \times \U)$. Combining this with
  $i \R^d \subset \U$ and \cite[Theorem~2.16(i)]{duffie2003} yields
  \eqref{eq:momentsPolyn}. By \cite[Lemma~A.1]{duffie2003}, for any
  $t \in [0,T]$, $ \E_x[|X_t|^{2 k}] $ is a sum of partial derivatives
  (up to order $k$) of $\psi(t,\cdot)$ and $\phi(t,\cdot)$ at $0$. But
  all of these are continuous (as argued above) and so
  $t \mapsto \E_x[|X_t|^{2 k}] $ is bounded on $[0,T]$. Hence
  \eqref{eq:momentsPolynIntegral} follows.
\end{proof}

\subsubsection{Time-inhomogeneous affine processes}
As it turns out, linear filtering of an affine process gives rise to a
time-inhomogeneous affine process. This class of time-inhomogeneous
Markov processes has been studied in \cite{Filipovic2005}. Similar to
the time-homogeneous case (as summarized in Section~\ref{sec:affine}),
\cite{Filipovic2005} has obtained characterizations in terms of a
martingale problem or (for conservative processes) semimartingale
characteristics. We do not repeat these here; for our purposes it is
sufficient to understand the conditions on the parameters that are
necessary and sufficient for the existence of such a process. For more
details we refer to \cite{Filipovic2005}.

A collection of parameters (depending on $t \geq 0$)
\begin{equation}\label{eq:parametersInhomogeneous}
  (a(t),\alpha(t),b(t),\beta(t),c(t),\gamma(t),\mu^0(t),\mu(t))
\end{equation}
is called admissible (or strongly admissible), if the following
(admissibility) conditions are satisfied:
\begin{itemize}
\item for any $t \geq 0$, \eqref{eq:parametersInhomogeneous} satisfies
  conditions \eqref{eq:admiss1}-\eqref{eq:admissLevyMes2},
\item $(a(t),\alpha(t),b(t),\beta(t),c(t),\gamma(t))$ are continuous
  in $t \in \R_+$,
\item the measures $\chi_k(\cdot)\mu^i(t,\cdot)$ (on
  $D \setminus \{0\}$) are weakly continuous in $t \in \R_+$ for any
  $ i \in I \cup \{0\}, \, k\in I \setminus \{i\}$,
\item the measures $\chi_k(\cdot)^2 \mu^i(t,\cdot)$ (on
  $D \setminus \{0\}$) are weakly continuous in $t \in \R_+$ for any
  $ i \in I \cup \{0\}, \, k \in (J \cup \{i\})\setminus\{0\}$.
\end{itemize}

\begin{remark} As before, $b$ and $\beta$ here denote
  $\tilde{b},\tilde{\beta}$ in
  \cite[Theorem~2.13]{Filipovic2005}. Since $\chi_k$ is bounded and
  continuous, the third continuity condition guarantees that
  $\tilde{b},\tilde{\beta}$ in \cite[Theorem~2.13]{Filipovic2005} are
  continuous if and only if $b$ and $\beta$ in
  \cite[Definition~2.5]{Filipovic2005} are continuous. Together with
  Remark~\ref{rmk:admissibility} this implies that the present
  admissibility conditions are identical with
  \cite[Definition~2.5]{Filipovic2005}.
\end{remark}

\begin{remark}
  If $c(t)=0$, $\gamma(t)=0$ for all $t \geq 0$, then the
  admissibility condition here is equivalent to
  \cite[Definition~2.4]{Kallsen2010}.
\end{remark}

By \cite[Theorem~2.13, Lemma~3.1 and Proposition~4.3]{Filipovic2005}
for any collection of parameters satisfying these conditions (and only
under these), there exists a strongly regular time-inhomogeneous
affine process $(\bar{X},(\P_{(r,x)})_{(r,x) \in \R_+ \times D})$ (a
time-inhomogeneous, stochastically continuous Markov process with an
additional regularity condition as \eqref{eq:affineRegular}, see
\cite{Filipovic2005}) with transition function $(P_{t,T})$ satisfying
for any $u \in \U$, $0\leq t \leq T$,
\begin{equation}\label{eq:affineDefInhomogeneous}
  P_{t,T}\exp(\langle u,\cdot \rangle)(x) = \exp(\Phi(t,T,u) + \langle x, \Psi(t,T,u) \rangle), \quad \forall x \in D,
\end{equation} 
where $\Phi$ and $\Psi$ solve the generalized Riccati equations
\begin{equation}\label{eq:genRiccatiInhomogeneous}\begin{aligned}
    - \partial_t \Phi(t,T,u) & = F(t,\Psi(t,T,u)), \quad \Phi(T,T,u) = 0 \\
    \partial_t \Psi(t,T,u) & = R(t,\Psi(t,T,u)) , \quad \Psi(T,T,u) =
    u, \quad 0 \leq t \leq T
  \end{aligned}
\end{equation}
with vector fields
\begin{equation}\label{eq:FandRInhomogeneous} \begin{aligned}
    F(t,u)&=  \frac{1}{2}\langle u, a(t) u \rangle + \langle b(t) , u \rangle -c(t) + \int_{D \setminus \{0\}}\left(e^{\langle \xi, u \rangle } - 1 - \langle \chi(\xi), u \rangle \right) \mu^0(t,\d \xi) \\
    R_i(t,u) & = \frac{1}{2}\langle u, \alpha^i(t) u \rangle + \langle
    \beta^i(t) , u \rangle -\gamma_i(t) \\ & \quad + \int_{D \setminus
      \{0\}}\left(e^{\langle \xi, u \rangle } - 1 - \langle \chi(\xi),
      u \rangle \right) \mu^i(t,\d \xi), \quad i = 1,\ldots, m, \\
    R_i(t,u) & = \langle \beta^i(t) , u \rangle, \quad i =
    m+1,\ldots,d,
  \end{aligned}
\end{equation}
where $\beta^i_j(t):=\beta_{j,i}(t)$.

Finally, fix $(r,x) \in \R_+ \times D$. As noted in
\cite{Filipovic2005} one may assume that $\bar{X}$ has RCLL paths,
$\P_{(r,x)}$-a.s. and so the following terminology makes sense:
Suppose $Y$ is a stochastic process on $(\Omega,\mathcal{F},\P)$ with
RCLL paths. We will say that (under $\P$) $Y$ is a time-inhomogeneous
affine process started in $(r,x)$ with admissible parameters
\eqref{eq:parametersInhomogeneous}, if the law of $Y$ under $\P$ (on
the space of RCLL-paths) is identical to the law of $\bar{X}$ under
$\P_{(r,x)}$.

\subsection{The filtering problem}
\subsubsection{Problem formulation and the Zakai
  equation}\label{subsec:filtering}
Fix $\pi_0 \in \Pm(D)$ and suppose $X$ is an affine process started
from $\pi_0$ (see Section~\ref{sec:affine}) on
$(\Omega,\mathcal{F},\P)$. Further, suppose $\F$ is a right-continuous
filtration on $(\Omega,\mathcal{F},\P)$ with respect to which $X$ is
adapted and such that $\mathcal{F}_0$ contains all $\P$-nullsets.

Let us introduce the problem of filtering $X$ given noisy observations
$Y$, as in the standard setup, see \cite{shiryaev2001} and
\cite{Bain2009}. The exposition here follows \cite{Kurtz1988}.

Define $Y$ as
\begin{equation} Y_t = \int_0^t h(X_s) \dd s + W_t, \quad t \geq 0,
\end{equation}
where $W$ is a $p$-dimensional $\F$-Brownian motion independent of
$X$, $h\colon D \to \R^p$ is measurable and
\begin{equation}\label{eq:secondmoment} \E\left[\int_0^T |h(X_s)|^2
    \dd s \right] < \infty, \end{equation}
for all $T \geq 0$. Set
\begin{align}\label{eq:obfil} \mathcal{F}_t^Y = \sigma(Y_s \; : \; 0
  \leq s \leq t) \vee \mathcal{N}, \quad t \geq 0,\end{align}

where $\mathcal{N}$ denotes the collection of $\P$-nullsets of
$(\Omega,\mathcal{F})$.

The goal of filtering theory is to calculate, for $t \in [0,\infty)$,
the conditional distribution of $X_t$ given
$\mathcal{F}_t^Y$. Formally this is described by a measure-valued
process as follows: By \cite[Theorem~2.1]{Bain2009} there exists a
$\Pm(D)$-valued $(\mathcal{F}_t^Y)_{t \geq 0}$-adapted, RCLL-process
$(\pi_t)_{t \geq 0}$ such that for any $f \in B(D)$, $t \geq 0$,
\[ \pi_t f = \E[f(X_t)|\mathcal{F}^Y_t] \quad \P\text{-a.s.} \] It can
be shown that $\pi$ satisfies the \textit{Kushner-Stratonovich
  equation}. This is a stochastic partial differential equation for
the process $\pi$, usually written in weak form, i.e. applied to test
functions $f \in D(\A)$.

Alternatively, one may consider an $\M(D)$-valued (but not
$\Pm(D)$-valued) process, which leads to the linear
Duncan-Mortensen-Zakai equation or shortly \textit{Zakai equation}:
Define
\begin{equation}\label{eq:sigmaDefAff}
  \sigma_t :=
  \exp\left(\int_0^t
    (\pi_s h)^\top \dd
    Y_s -
    \frac{1}{2}\int_0^t|\pi_s
    h|^2 \dd s\right)
  \pi_t \end{equation}
which is nonzero $\P$-a.s., for any $t \geq 0$, because
\begin{equation}\label{eq:hPiIntegrable}
  \E\left[\int_0^t |\pi_s h|^2 \dd s \right] < \infty,
\end{equation}
as can be deduced from \eqref{eq:secondmoment}.

We are now concerned with the filtering problem on the time interval
$[0,T]$, for some $T > 0$ fixed. By \eqref{eq:secondmoment} and
independence,
\begin{equation}\label{eq:changeOfMeasure}
  \frac{\d \Q}{\d \P} = \exp\left(-\int_0^T h(X_s)^\top \dd W_s - \frac{1}{2} \int_0^T |h(X_s)|^2 \dd s  \right)
\end{equation} 
defines a new probability measure $\Q$ on $(\Omega,\mathcal{F})$ that
is equivalent to $\P$ on $\mathcal{F}_T$.\footnote{See
  \cite[Example~I.6.2.4]{shiryaev2001}. Independence is crucial here,
  otherwise a Novikov' type assumption would be needed. } Furthermore,
the law of $X$ under $\P$ is the same as under $\Q$ and, on $[0,T]$
under the measure $\Q$, $Y$ is a Brownian motion independent of $X$.

It can be shown (see \cite[Exercise ~3.37]{Bain2009}) that
$\sigma_t 1$ defined in \eqref{eq:sigmaDefAff} is equal to
$ \E_\Q[\frac{\d \P}{\d \Q}|\mathcal{F}^Y_t]$. Combining this with the
abstract Bayes' rule and the definition \eqref{eq:sigmaDefAff}, one
obtains (see \cite[Proposition~3.16]{Bain2009}) that for any
$t \in [0,T]$, $f \in B(D)$,
\begin{equation}\label{eq:sigma}
  \sigma_t f =
  \E_\Q\left[\left. f(X_t)\exp
      \left(\int_0^t
        h(X_s)^\top \dd
        Y_s -
        \frac{1}{2}
        \int_0^t
        |h(X_s)|^2 \dd s
      \right)
    \right|\mathcal{F}^Y_t\right], \end{equation}
$\P$-a.s., and the \textit{Kallianpur-Striebel} formula
\begin{equation}\label{eq:KallianpurStriebel}
  \pi_t f =
  \frac{\sigma_t
    f}{\sigma_t
    1}. \end{equation}
Furthermore, $\sigma$ satisfies the \textit{Zakai equation} 
\begin{equation}\label{eq:Zakai}
  \sigma_t f = \pi_0 f + \int_0^t \sigma_s(\A f) \dd s + \int_0^t \sigma_s(h f) \dd Y_s \quad \text{for any } f \in D(\A).
\end{equation}
By \eqref{eq:hPiIntegrable} and \eqref{eq:sigmaDefAff}, $h$ is
$\sigma_t$-integrable for all $t \leq T$ and
$\int_0^T |\sigma_t h|^2 \dd s < \infty$, $\P$-a.s. Hence all terms in
\eqref{eq:Zakai} are indeed well-defined.

\subsubsection{Uniqueness for the Zakai equation}

The following result is a consequence of \cite[Theorem
4.2]{Kurtz1988}:

\begin{theorem}[Well-posedness of the Zakai equation]\label{thm:zakaiUniqueness}
  Let $h \in C(D)$, $\A$ the generator of a (conservative) affine
  process (see \eqref{eq:affineGeneratorShort}) and $\sigma$ as in
  \eqref{eq:Zakai}.  Assume \eqref{eq:secondmoment}.

  Suppose $(\rho_t)_{t \in [0,T]}$ is an
  $(\mathcal{F}^Y_t)_{t\in [0,T]}$-adapted RCLL $\M(D)$-valued process
  such that $h$ is $\rho_t$-integrable for all $t \leq T$,
  $\int_0^T |\rho_t h|^2 \dd s < \infty$, $\P$-a.s. and satisfying
  \begin{equation}
    \rho_t f = \pi_0 f + \int_0^t \rho_s(\A f) \dd s + \int_0^t \rho_s(h f) \dd Y _s, \quad \text{for any } f \in C_c^\infty(D)
  \end{equation}
  and for $f =1$ (with $\A 1:= 0$). Then $\rho_t = \sigma_t$ for all
  $t < T$, $\P$-a.s.
\end{theorem}
\begin{proof}
  Define $D(\A_0):=C_c^\infty(D)$ and $\A_0$ the restriction of $\A$
  to $D(\A_0)$. Then by Lemma~\ref{lem:martingaleProblem}, for any
  $\pi_0 \in \Pm(D)$, the martingale problem for
  $(D(\A_0),\A_0,\pi_0)$ is well-posed. Furthermore, for any
  $f \in D(\A_0)$, $h_i f \in C_0(D)$ for $i = 1,\ldots, p$ and so the
  assumptions of \cite[Theorem 4.2]{Kurtz1988} are indeed satisfied.
\end{proof}
Since the assumptions for \cite[Theorem 4.1]{Kurtz1988} are the same
as for \cite[Theorem 4.2]{Kurtz1988}, as a corollary one also obtains
a uniqueness result for the Kushner-Stratonovich equation.

Let us point out that Theorem~\ref{thm:zakaiUniqueness} holds in the
setting considered in Section~\ref{sec:linearizedFiltering}. Taking
$h(x)=x$, $\pi_0 = \delta_x$ for some $x \in D$ and assuming that the
jump-measures of the affine process satisfy \eqref{eq:expMoments}, one
obtains from Lemma~\ref{lem:affMoments} that \eqref{eq:secondmoment}
is indeed satisfied.

\subsubsection{Robust filtering}
Thanks to the uniqueness result for the Zakai equation in
Theorem~\ref{thm:zakaiUniqueness}, theoretically the filtering problem
is settled: One finds a solution to the Zakai equation and uses the
Kallianpur-Striebel formula \eqref{eq:KallianpurStriebel} to calculate
the filter. However, in practice one is given a fixed
$y \in C([0,T],\R^p)$ (of \textit{finite variation}), whereas
\eqref{eq:sigma} only specifies the filter $\P$-a.s. Thus a definition
of \eqref{eq:sigma} for \textit{all} $y \in C([0,T],\R^p)$ is needed.

Let us briefly review the main result of \cite{Davis1980}.  See
\cite[Chapter 5]{Bain2009} and \cite[Section~1.4]{vanHandel2007} for
further references on robust filtering. Suppose $h \in D(\A)$ so that
$h(X)$ is a semimartingale. Since $X$ and $Y$ are independent, one can
integrate by parts
\begin{equation} \label{eq:productRule} \int_0^t h(X_s)^\top \dd Y_s =
  Y_t^\top h(X_t) - \int_0^t Y_s^\top \dd h(X_s)
\end{equation}
and rewrite $\sigma$ in \eqref{eq:sigma} as
\[ \sigma_t f = \E_\Q\left[\left. f(X_t)\exp \left(Y_t^\top h(X_t) -
        \int_0^t Y_s^\top \dd h(X_s)- \frac{1}{2} \int_0^t |h(X_s)|^2
        \dd s \right) \right|\mathcal{F}^Y_t\right].  \] Recalling
that $X$ and $Y$ are independent under $\Q$ and $X$ has the same
distribution under $\P$ as under $\Q$, the conditional expectation is
actually given as $ \E\left.[F(X,y)]\right\rvert_{y=Y} $ for a
suitable function $F \colon D \times \R^p \to \mathbb{R}$. In fact,
the following \textit{robustness} property has been established in
\cite{Clark1978}, \cite{Clark2005}: Define the \textit{pathwise
  filtering functional}
$\sigma_t \colon B(D) \times C([0,T],\R^d) \to \mathbb{R}$ by
\begin{equation}\label{eq:pathwiseFiltering}\sigma_t(f,y)=\E\left[f(X_t)\exp
    \left(y_t^\top h(X_t) - \int_0^t y_s^\top \dd h(X_s)- \frac{1}{2}
      \int_0^t |h(X_s)|^2 \dd s \right)\right], \end{equation}
then $\sigma_t(f,\cdot)/\sigma_t(1,\cdot)$ is locally Lipschitz continuous and
\begin{equation}\label{eq:pathwiseFiltering2}
  \frac{\sigma_t(f,Y)}{\sigma_t(1,Y)} = \E[f(X_t)|\mathcal{F}_t^Y], \quad \P\text{-a.s.} \end{equation}
See also \cite{Crisan2013} for an extension to multidimensional observation and correlated noise.

In \cite{Davis1980} the following observation is made: Fix
$y \in C[0,T]$ and define a two-parameter semigroup of operators on
$B(D)$ by
\begin{equation}\label{eq:davissemigroup} T_{s,t}^y
  f(x)=\E_{x}\left[f(X_{t-s})\exp\left(-\int_s^t y_u^\top \dd
      h(X_{u-s}) - \frac{1}{2} \int_s^t |h(X_{u-s})|^2 \dd u
    \right)\right], \end{equation}
for $t \geq s \geq 0$, $x \in D$.
Then 
\begin{equation}\label{eq:davisresult} \sigma_t(f,y) = \int_D
  T_{0,t}^y(e^{y(t)h}f)(x) \pi_0(\d x) \end{equation}
and, this is the main result of \cite{Davis1980}, the (extended) generator $\mathcal{A}^y_t$ of the semigroup $T_{s,t}^y$ is given by
\[ \A_t^y f = e^{y(t)h}(\A - \frac{1}{2}h^2)(e^{-y(t)h}f). \]

This is closely related to applying a Doss-Sussmann method (see
e.g. \cite[Theorem 28.2]{Rogers2000}) to the Zakai equation, as
explained in \cite{Davis2011}.

\section{The linearized filtering
  functional}\label{sec:linearizedFiltering}

In this section we introduce and study a computationally tractable
approximation of the pathwise filtering functional
\eqref{eq:pathwiseFiltering} when both $X$ and $h$ are
affine. Throughout this section $X$ is an affine process on
$(\Omega,\mathcal{F},\P)$ started from $\pi_0 \in \Pm(D)$ with
admissible parameters \eqref{eq:collectionOfParameters} and $F$, $R$
are as in \eqref{eq:FandR}. If $\pi_0 = \delta_{x}$ for $x \in D$, we
write $\P_x$ for $\P$.

\subsection{Definition and main results}
\subsubsection{Definition of the approximate filter}
Fix an observation $y \in C([0,\infty),\R^d)$ with $y(0)=0$ and
functions $\gamma \in C([0,\infty),\R^d)$, $c \in
C([0,\infty),\R)$. The \textit{linearized filtering functional} (LFF)
$\rho$ is defined as
\begin{equation}\label{eq:modifiedZakaiFilter}\begin{aligned}
    \rho_t(f,y)& =\E \left[f(X_t)\exp \left(y_t^\top X_t - \int_0^t
        y_s^\top \dd X_s - \int_0^t \gamma_s^\top X_s - c_s \dd s
      \right) \right]
  \end{aligned}
\end{equation}
for any $t \geq 0$ and $f\colon D \to \R$ measurable such that the
right hand side of \eqref{eq:modifiedZakaiFilter} is well-defined
(e.g. $f \geq 0$). If $\rho_t(1,y)$ is finite, define the approximate
pathwise filter (the \textit{affine functional filter} or AFF) by
\begin{equation}\label{eq:approximateFilter} \bar{\pi}_t(f,y)=
  \frac{\rho_t(f,y)}{\rho_t(1,y)}. \end{equation}
If $\pi_0 = \delta_{x}$ for $x \in D$, we write $ \rho^x_t(f,y)$ for $\rho_t(f,y)$ and $ \bar{\pi}_t^x(f,y)$ for $ \bar{\pi}_t(f,y)$.
\subsubsection{Heuristic motivation}
The linearized filtering functional \eqref{eq:modifiedZakaiFilter} is
the same as the pathwise filtering functional
\eqref{eq:pathwiseFiltering} for $h(x)=x$, but with $\frac{1}{2}|x|^2$
approximated by the affine function $\gamma_s x+c_s$. The motivation
for studying $\rho_t$ is the following: if for some $t>0$, $x_0 \in D$
and (small) $\varepsilon > 0$,
$\P(\{X_s \in B_\varepsilon(x_0) \, \forall s \in [0,t]\})$ is almost
$1$, then \eqref{eq:modifiedZakaiFilter} and
\eqref{eq:pathwiseFiltering} (with $\gamma_s = x_0$ and
$c_s=\frac{x_0^2}{2}$) are very close. Consequently,
\eqref{eq:pathwiseFiltering2} implies that also the approximate filter
$\bar{\pi}_t(f,Y)$ should be close to $\pi_t(f)$.
\subsubsection{Fourier filtering}
The key point is that \eqref{eq:modifiedZakaiFilter} is
computationally tractable, since one can calculate the Fourier
coefficients of \eqref{eq:modifiedZakaiFilter} by solving a system of
generalized Riccati equations:

\begin{theorem}\label{thm:genRiccatiFilter}
  Assume \eqref{eq:expMoments} holds. Let $u \in \C^d$ and
  $T \in \R_+$. Suppose $\Phi \in C^1([0,T],\R)$ and
  $\Psi \in C^1([0,T],\R^d)$ solve
  \begin{equation}\label{eq:genRiccatiFilter}
    \begin{aligned}
      - \partial_t \Phi(t,T,u) & = F(\Psi(t,T,u)-y_t)-c_t, \quad \Phi(T,T,u) = 0 \\
      -\partial_t \Psi(t,T,u) & = R(\Psi(t,T,u)-y_t)-\gamma_t , \quad
      \Psi(T,T,u) = u+y_T, \quad 0 \leq t \leq T.
    \end{aligned}
  \end{equation}
  Then for any $x \in D$, the Fourier coefficient of
  $\rho_T^x(\cdot,y)$ is well-defined and given as
  \begin{equation}\label{eq:FourierCoeff} \rho^x_T(\exp(\langle u ,
    \cdot \rangle),y) = \exp(\Phi(0,T,u)+\langle x, \Psi(0,T,u)\rangle
    ). \end{equation}
  Furthermore, there exists $T_0>0$ such that for all $u \in i \R^d$ and $T\leq T_0$, the system \eqref{eq:genRiccatiFilter} has a unique solution on $[0,T]$.
\end{theorem}

The proof of Theorem~\ref{thm:genRiccatiFilter} is postponed to
Section~\ref{sec:mainproofs} below. Let us briefly discuss how to use
Theorem~\ref{thm:genRiccatiFilter} in practice, relate it to the
literature and discuss its assumptions.

\begin{remark} Suppose $f \colon D \to \C$ is given as
  \[ f(y) = \int_{\R^d} e^{i \langle v, y \rangle} \hat{f}(v) \dd v,
    \quad y \in D \] for some $\hat{f} \colon \R^d \to \C$
  integrable. Then for any $T > 0$ small enough, by
  Theorem~\ref{thm:genRiccatiFilter}, definition
  \eqref{eq:modifiedZakaiFilter} and Fubini's theorem
  \[ \rho^x_T(f,y) = \int_{\R^d} \rho^x_T(\exp(\langle iv , \cdot
    \rangle),y) \hat{f}(v) \dd v = \int_{\R^d} e^{\Phi(0,T,i
      v)+\langle x, \Psi(0,T,i v)\rangle}\hat{f}(v) \dd v. \] This is
  analogous to the Fourier method used in option pricing in the
  framework of affine models.
\end{remark}

\begin{remark} Expressions of type \eqref{eq:FourierCoeff} are called
  affine transform formulas in the literature, see
  e.g. \cite{keller-ressel2015} and the references therein. Note that
  the present result is not covered in the literature, since the
  Riccati equations \eqref{eq:genRiccatiFilter} are time-inhomogeneous
  and correspond to a non-conservative affine ``process'' for which
  the admissiblity conditions \eqref{eq:ccond} and \eqref{eq:gamcond}
  are not necessarily satisfied.
\end{remark}

\begin{remark} In general, it does not hold that
  $\rho_T(1,y) < \infty$ for all $T > 0$ and so the statement of
  Theorem~\ref{thm:genRiccatiFilter} really just holds up to a finite
  $T_0$ (depending on $y$). To see this, let $u \in \R^d$ and consider
  $y_s = u s$, $c_s = \gamma_s = 0$ for all $s \geq 0$. Then the
  product rule (as in \eqref{eq:productRule}) and $\dot y_s = u$ show
  \[ \rho^x_T(1,y)= \E_x\left[\exp\left(u^\top \int_0^T X_s \dd
        s\right)\right] \] which is not necessarily finite. For
  example, if $d=m=1$ and $X$ is a CIR process (see
  Section~\ref{sec:CIR}) with parameters $\beta > 0$, $b \geq 0$ and
  $\sigma > 0$, then for $u < \beta^2/(2 \sigma^2)$ and $T$ large
  enough (satisfying $\tanh(\gamma T /2)\geq \gamma / \beta $ with
  $\gamma = \sqrt{\beta^2-2\sigma^2 u}$) the expectation is not
  finite, see \cite{friz2010d} or \cite{dufresne2001}.

  Finally, note that \eqref{eq:expMoments} could be weakened to the
  following assumption: there exists $V \subset \R^d$ open with
  $0 \in V$ such that \eqref{eq:expMoments} holds for $u \in V$
  (instead of all $u \in \R^d$).

\end{remark}

\subsubsection{The smoothing distribution} Our approximation
\eqref{eq:modifiedZakaiFilter} and \eqref{eq:approximateFilter} also
gives rise to an approximation of the smoothing distribution, i.e. the
distribution of $X_{[0,t]}$ conditional on $\mathcal{F}^Y_t$.

Fix $t>0$ and denote by $D[0,t]$ the set of RCLL-mappings
$[0,t] \to D$.  Consider $G \colon D[0,t] \to \R$ bounded,
measurable\footnote{More precisely, for $s \in [0,T]$ define
  $Y_s \colon D[0,t] \to \R$ by $Y_s(\omega):=\omega(s)$ and equip
  $D[0,t]$ with the $\sigma$-algebra generated by
  $(Y_s)_{s \in [0,t]}$.} and, analogously to
\eqref{eq:modifiedZakaiFilter} and \eqref{eq:approximateFilter} define
\begin{equation}\label{eq:modifiedZakaiSmoother}\begin{aligned}
    \rho_t(G,y)& =\E \left[G(X_{[0,t]})\exp \left(y_t^\top X_t - \int_0^t y_s^\top \dd X_s - \int_0^t \gamma_s^\top X_s - c_s \dd s  \right) \right] \\
    \bar{\pi}_t(G,y) & = \frac{\rho_t(G,y)}{\rho_t(1,y)}
  \end{aligned}
\end{equation}
for any $t \geq 0$ such that $\rho_t(1,y) < \infty$. Then
$\bar{\pi}_t(\cdot,y)$ is a probability measure on $D[0,t]$ and an
approximation to the smoothing distribution. Again, if
$\pi_0 = \delta_{x}$ for $x \in D$, we write $ \rho^x_t(G,y)$ for
$\rho_t(G,y)$ and $ \bar{\pi}_t^x(G,y)$ for $ \bar{\pi}_t(G,y)$.

The following result shows that $\bar{\pi}_t(\cdot,y)$ coincides with
the the distribution on $D[0,t]$ of a time-inhomogeneous affine
process. It will be used for the calculation of (approximate)
conditional moments in Section~\ref{sec:CIR} and \ref{sec:Wishart}
below. To formulate it, define
\[\begin{aligned}
    \widehat{\pi_0}(z) = \int_D e^{\langle x, z \rangle} \pi_0(\d x) ,
    \quad \text{ for } z \in D_{\pi_0} = \{z \in \C^d \; : \;
    |\exp(\langle \cdot, z \rangle)| \in L^1(D,\pi_0)
    \}. \end{aligned}\]

\begin{theorem}\label{thm:timeInhomogeneousAffine}  Let $T_0 > 0$, $t \in (0,T_0]$ and $\Psi(\cdot,t,0)$ as in Theorem~\ref{thm:genRiccatiFilter}. Suppose $\Psi(0,t,0) \in D_{\pi_0}$.
  Then for any $G \in B(D[0,t])$,
  \[ \bar{\pi}_t(G,y) = \frac{\int_D e^{\langle x, \Psi(0,t,0)
        \rangle} \E_{\mathbb{Q}_{x}^{y,t}}[G(X_{[0,t]})] \pi_0(\d
      x)}{\widehat{\pi_0}(\Psi(0,t,0))},\] where under
  $\mathbb{Q}_{x}^{y,t}$, $X$ is a time-inhomogeneous affine process
  started from $(0,x)$ with admissible parameters
  \[ (a(s),\alpha(s),b(s),\beta(s),0,0,\mu^0(s),\mu(s))_{s \geq 0} \]
  defined by \eqref{eq:parametersFilter} below with
  $g(s):=\Psi(s\wedge t,t,0)-y_{s \wedge t}$ for $s \geq 0$.
\end{theorem}

\begin{remark}
  If $\pi_0 = \delta_{x}$ for $x \in D$, then $D_{\pi_0} = \C^d$,
  $\widehat{\pi_0}(z)= e^{\langle x, z \rangle}$ and so
  Theorem~\ref{thm:timeInhomogeneousAffine} implies
  \begin{equation}\label{eq:smoothingApprox} \bar{\pi}_t^x(G,y) =
    \E_{\mathbb{Q}_{x}^{y,t}}[G(X_{[0,t]})].\end{equation}
\end{remark}

\begin{remark} As a simple example, consider a CIR process (see
  Section~\ref{sec:CIR}) started in $x >
  0$. Theorem~\ref{thm:timeInhomogeneousAffine} implies that for
  $t \leq T_0$ the approximate smoothing distribution is given by
  \eqref{eq:smoothingApprox}. Under $\mathbb{Q}_{x}^{y,t}$ the process
  $X$ is the unique solution to
  \begin{equation}\label{eq:CIRNewdynamics}
    d X_s = b + \beta X_s +u(s,X_s)\d s + \sigma \sqrt{X_s} \d B_s , \quad X_0 = x,
  \end{equation}
  where $u(s,x):= \sigma^2 (\Psi_{s\wedge t}-y_{s \wedge t}) x$,
  $\Psi_s:=\Psi(s,t,0)$ solves (the second part of)
  \eqref{eq:genRiccatiFilter} and $B$ is a Brownian motion under
  $\mathbb{Q}_{x}^{y,t}$. Thus, the approximate smoothing distribution
  is the distribution (on path space) of a new process, which is
  obtained by inserting the additional drift term $u(s,X_s)$ in the
  original SDE \eqref{eq:CIRdynamics}.

  From \cite[Chapters~1.4.3 and 4.2]{vanHandel2007} one obtains
  formally a representation analogous to \eqref{eq:smoothingApprox}
  for the exact smoothing distribution, the only difference being the
  choice of $u$ in \eqref{eq:CIRNewdynamics}. However, calculating the
  function $u$ in this case requires solving a PDE. For the
  approximate filter $u$ can be obtained by solving an ODE, which is
  an enormous reduction of complexity.
\end{remark}

\subsubsection{An alternative point of view} To clarify further the
relation to \cite{Davis1980}, let $x_0 \in \R_{++}^m \times \R^n$,
$c_0 > 0$ and define $H(x) = (x_0)_I^\top x_I + c_0$ and
\begin{equation}\label{eq:modifiedZakaiFunctional} \bar{T}_t^y
  f(x)=\E_{x}\left[f(X_t)\exp\left(-\int_0^t y_u^\top \dd X_u -
      \int_0^t H(X_u) \dd u \right) \right], \end{equation}
so that, in analogy to \eqref{eq:davisresult} it holds that (with $\gamma^\top = ((x_0)_I^\top,0))$ and $c = c_0$ in \eqref{eq:modifiedZakaiFilter})
\[ \rho_t^x(f,y) = \bar{T}_t^y(\exp(\langle y_t ,\cdot
  \rangle)f)(x). \] Thus we have approximated $T_{0,t}^y$ in
\eqref{eq:davissemigroup} by $\bar{T}_t^y$. Now if $\beta^i=0$ for
$i \in J$, then \eqref{eq:modifiedZakaiFunctional} corresponds to a
non-conservative, time-inhomogeneous affine process:

\begin{proposition}\label{thm:relationToDavis} Assume \eqref{eq:expMoments} holds and $\beta^i = 0$ for $i \in J$. Then there exists $c_0>0, T> 0$ such that for all $t \in [0,T]$, $\bar{T}_t^y$ in \eqref{eq:modifiedZakaiFunctional} satisfies
  $\bar{T}_t^y = P_{0,t}^y$, where $(P_{s,t}^y)$ is the transition
  semigroup of a time-inhomogeneous affine process with (admissible)
  parameters \eqref{eq:parametersInhomogeneous} defined for all
  $t \geq 0$ by \eqref{eq:parametersFilter} below with
  $g(t)=-y_{t\wedge T}$ and
  \begin{equation}
    \begin{aligned}
      c(t) & =  c_0 - F(-y_t) \\
      \gamma^i(t) & = x_0^i - R_i(-y_t) , \quad i \in I.
    \end{aligned}
  \end{equation}
\end{proposition}

\subsubsection{Discussion}

\begin{remark} The ordinary differential equation
  \eqref{eq:genRiccatiFilter} is formulated backwards in time, which
  appears to lead to a non-recursive filter. This can easily be
  resolved and we now explain how a recursive procedure can be
  obtained: Fix $T_0 > 0$ sufficiently
  small. Theorem~\ref{thm:genRiccatiFilter} guarantees that for any
  $v \in \R^d$ and $T \leq T_0$ there exists a unique $u_0 \in \C^d$
  such that the ODE
  \begin{equation}
    \label{eq:genRiccatiFilterRecursive} \begin{aligned}
      -\partial_t \bar{\Psi}(t,u_0) & = R(\bar{\Psi}(t,u_0)-y_t)-\gamma_t  \\   \bar{\Psi}(0,u_0) & = u_0 \end{aligned}
  \end{equation}
  has a unique solution on $[0,T]$ with
  $\bar{\Psi}(T,u_0)=iv+y_T$. More specifically, one chooses
  $u_0:=\Psi(0,T,iv)$ and $\bar{\Psi}(t,u_0):=\Psi(t,T,iv)$. This
  gives the following recursive procedure to calculate the approximate
  filter at time $T \leq T_0$:
  \begin{itemize}
  \item solve for all $u_0 \in \C^d$ (for which a solution exists) the
    ODE \eqref{eq:genRiccatiFilterRecursive} up to time $T$.
  \item for $v \in \R^d$, find the unique solution $u_0 \in \C^d$ to
    $\bar{\Psi}(T,u_0) = iv+y_T $ and evaluate
    \[ \rho^x_T(\exp(\langle iv , \cdot \rangle),y) = \exp \left(
        \int_0^T F(\bar{\Psi}(s,u_0)-y_s)-c_s \dd s+\langle x, u_0
        \rangle \right). \]
  \end{itemize}
  In order to calculate the approximate filter at time
  $\tilde{T} \in [T,T_0]$ one only needs to continue solving
  \eqref{eq:genRiccatiFilterRecursive} on $[T,\tilde{T}]$ (and then
  repeat the second step for $v \in \R^d$), hence the procedure is
  indeed recursive.
\end{remark}

\begin{remark}\label{rmk:obsNoiseScale} Consider a $p$-dimensional
  Brownian motion $W$, $C \in \R^{p \times d}$,
  $\Gamma \in \R^{p \times p}$ invertible and an observation process
  given as
  \begin{equation}\label{eq:YBar} \bar{Y}_t = \int_0^t C X_s \dd s +
    \Gamma W_t, \quad t \geq 0.
  \end{equation}
  The present methodology also provides an approximation for this
  setup: Since $\bar{Y}$ and $Y =\Gamma^{-1} \bar{Y}$ generate the
  same filtration, the filtering distribution is given by
  \eqref{eq:pathwiseFiltering} and \eqref{eq:pathwiseFiltering2} with
  $h(x)=\Gamma^{-1} C x$. The pathwise functional $\sigma_t(f,y)$ in
  \eqref{eq:pathwiseFiltering} is approximated naturally by
  $\rho_t(f, (\Gamma^{-1} C)^\top y))$ (see
  \eqref{eq:modifiedZakaiFilter}) with
  $\gamma_s = (\Gamma^{-1} C)^\top \Gamma^{-1} C x_0$ (corresponding
  to the linearization of $h$ around $x_0$) and
  $x_0 = \int_D x \pi_0(\d x)$. We do not specify $c$ here, since it
  cancels out in the normalization \eqref{eq:approximateFilter}.

  In fact, this choice of $\gamma$ has been used in the examples in
  Section~\ref{sec:CIR}.
\end{remark}

\begin{remark}
  The choice of the functions $\gamma$ and $c$ is of course essential
  for how close $\rho_t$ and $\bar{\pi}_t$ are to $\sigma_t$ and
  $\pi_t$.  In the examples we have always made the choice specified
  in the previous remark. Let us examine the approximation quality in
  this setting. Using the product rule \eqref{eq:productRule}, the
  definition of $\Q$ and applying the change of measure
  \eqref{eq:changeOfMeasure} in \eqref{eq:modifiedZakaiFilter} yields
  $\P$-a.s.
  \[\begin{aligned} \rho_t(f,Y)&= \E_\Q\left[\left. f(X_t)\frac{\d \P}{\d \Q}\exp \left(\frac{1}{2} \int_0^t |h(X_s)|^2 \dd s - \int_0^t \gamma_s^\top X_s - c_s \dd s  \right) \right|\mathcal{F}^Y_t\right]\\
      & = \E \left[\left. f(X_t) \exp \left(\frac{1}{2} \int_0^t
            |h(X_s)|^2 \dd s - \int_0^t \gamma_s^\top X_s - c_s \dd s
          \right) \right|\mathcal{F}^Y_t\right] \sigma_t 1
    \end{aligned}\] and so (in the setting of the previous remark)
  \[ \bar{\pi}_t(f,Y) = \frac{\E \left[\left. f(X_t) A_t
        \right|\mathcal{F}^Y_t\right]}{\E \left[\left. A_t
        \right|\mathcal{F}^Y_t\right]} , \quad A_t = \exp
    \left(\frac{1}{2} \int_0^t |\Gamma^{-1} C(X_s-x_0)|^2 \dd s
    \right). \] This gives an indication about the approximation
  quality:
 
  If (with high probability) $\log A_t$ is very small, then the
  approximation quality is good. This happens for example if
  $\Gamma = \varepsilon I$ for large $\varepsilon > 0$. If
  $\varepsilon > 0$ is very small on the other hand, then the
  approximation quality decreases. However, in this regime there is no
  need for filtering, since $\int_0^\cdot X_s \dd s$ can be almost
  read off from \eqref{eq:YBar}. For intermediate values of
  $\varepsilon$ this is more difficult to judge and from numerical
  experiments it appears that there is a range of $\varepsilon$ for
  which the filtering problem is not easy, and nevertheless the
  approximation is not very good.
\end{remark}

\begin{remark} If the observations arrive only at discrete-time points
  (as opposed to the cont\-inuous-time setting considered here) a
  similar approximation can be defined. In this case the ordinary
  differential equations \eqref{eq:genRiccatiFilter} are replaced by
  difference equations.
\end{remark}

\section{Proofs}
\label{sec:proofsAff}

\subsection{Proof of auxiliary results}

In this section we prepare for the proof of the main results. To this
end, we study a change of measure, estimates for the function $R$ in
\eqref{eq:FandR} and properties of $\bar{T}^y$ in
\eqref{eq:modifiedZakaiFunctional}.

\subsubsection{Change of measure}

One of the key tools in the proofs is a change of measure, which turns
the original (time-homogeneous) affine process into a
time-inhomogeneous affine process. The next Lemma~\ref{lem:admissible}
verifies that the associated parameters satisfy the admissibility
conditions. Based on this, Proposition~\ref{lem:truemartingale} below
will then provide the ingredients for the change of measure.

\begin{lemma}\label{lem:admissible} Suppose $g \colon \R_+ \to \R^d$ is continuous, \eqref{eq:collectionOfParameters} are admissible with $c=0$, $\gamma = 0$ and \eqref{eq:expMoments} holds.
  For $t \geq 0$, define parameters \eqref{eq:parametersInhomogeneous}
  by $c(t)=0$, $\gamma(t)=0$ and
  \begin{equation} \label{eq:parametersFilter}
    \begin{aligned}
      a(t) & = a \\
      \alpha(t) & = \alpha \\
      b(t) & = b + a g_t + \int_{D \setminus \{0\}} \chi(\xi)(e^{\langle g_t,\xi\rangle}-1)\mu^0(\d \xi) \\
      \beta_{i,j}(t) & = \beta_{i,j} + (\alpha^j g_t)_i + \int_{D \setminus \{0\}} \chi_i(\xi)(e^{\langle g_t,\xi\rangle}-1)\mu^j(\d \xi) , \quad i \in I \cup J, \, j \in I   \\
      \beta_{i,j}(t) & = \beta_{i,j}, \quad i \in I \cup J, \, j \in J \\
      \mu^i(t,\d \xi) & = e^{\langle g_t,\xi\rangle} \mu^i(\d \xi),
      \quad i \in I \cup \{0\}.
    \end{aligned}
  \end{equation}
  Then \eqref{eq:parametersInhomogeneous} is strongly admissible and
  for all $T \geq 0$,
  \begin{equation}\label{eq:KallsenMuhleKarbeCond}
    \sup_{t \in [0,T]} \int_{\{\xi_k > 1\}} \xi_k \mu^i(t,\d \xi) < \infty,\quad \text{ for } \quad i,k \in I.
  \end{equation}
\end{lemma}
\begin{proof}
  \textbf{Admissibility for fixed $t \geq 0$:} Firstly,
  \eqref{eq:admiss1} implies $a_{i,j}=0$ for all
  $i \in I, j \in I \cup J$ (see (2.4) in \cite{duffie2003}). Thus,
  for $i \in I$, definition \eqref{eq:parametersFilter}, the assumed
  integrability \eqref{eq:admissLevyMes1} and the non-negativity
  condition \eqref{eq:driftcond1} yield
  \begin{equation*} b_i(t) - \int_{D \setminus \{0\}} \chi_i(\xi)
    \mu^0(t,\d \xi)= b_i - \int_{D \setminus \{0\}} \chi_i(\xi)
    \mu^0(\d \xi) \geq 0.\end{equation*} Similarly, for $i,j \in I$
  with $i \neq j$, \eqref{eq:alphacond} implies $\alpha^j_{i,k}=0$ for
  all $k \in I \cup J$. If this was not the case, i.e. if
  $\alpha^j_{i,k}\neq 0$ for some $k \in J \cup \{j\}$, then defining
  $v \in \R^d$ by $v_l = \delta_{l k}$ for $l \in J \cup \{j\}$,
  $v_l = C \delta_{l i}$ for $l \in I \setminus \{j\}$ and using
  \eqref{eq:alphacond} would yield
  \[ 0 \leq v^\top \alpha^j v = 2 C \alpha^j_{i,k} + \alpha^j_{k,k} \]
  for all $C \in \R$ and hence a contradiction. Consequently
  $(\alpha^j g_t)_i = 0$ and as above one uses \eqref{eq:driftcond2}
  and \eqref{eq:admissLevyMes1} to obtain
  \begin{equation*}
    \beta_{i,j}(t)-  \int_{D \setminus \{0\}} \chi_i(\xi) \mu^j(t,\d \xi) = \beta_{i,j} - \int_{D \setminus \{0\}} \chi_i(\xi) \mu^j(\d \xi) \geq  0.\end{equation*}
  Finally, for $i \in I \cup \{0\}$ and any non-negative $f \in B(D)$  one uses $|e^{\langle g_t,\xi\rangle}|\leq e^{|g_t|}$ on $\{|\xi| \leq 1\}$ to estimate
  \begin{equation}\label{eq:auxEqAff14} \int_{D \setminus \{0\} }
    f(\xi) \mu^i(t,\d \xi) \leq e^{|g_t|}\int_{\{|\xi| \leq 1\}
      \setminus \{0\}} f(\xi) \mu^i(\d \xi) + \|f\|_\infty \int_{D
      \setminus \{|\xi| \leq 1\}} |\xi| e^{\langle g_t,\xi\rangle}
    \mu^i(\d \xi). \end{equation}
  Inserting $f=\chi_k$ for $k \in I \setminus \{i\}$ and $f=\chi_k^2$ for $k \in (J \cup \{i\})\setminus\{0\}$ in \eqref{eq:auxEqAff14}, the integrability conditions for $\mu^i(t,\cdot)$ follow from \eqref{eq:admissLevyMes1}, \eqref{eq:admissLevyMes2} and \eqref{eq:expMoments}.

  Altogether, it has been verified that
  \eqref{eq:parametersInhomogeneous} satisfy for each $t \geq 0$
  conditions \eqref{eq:admiss1}-\eqref{eq:admissLevyMes2}.

  \textbf{Continuity in $t$:} Let us first verify the third and fourth
  admissibility conditions. To do so, note that for any
  $f \colon D\setminus \{0\} \to \R$ which is $\mu^i$-integrable,
  dominated convergence and continuity of $g$ yield that
  \begin{equation}\label{eq:auxEqAff2} t \mapsto \int_{\{|\xi| \leq
      1\} \setminus \{0\}} f(\xi) e^{\langle g_t,\xi\rangle} \mu^i(\d
    \xi) \quad \text{is continuous.} \end{equation}
  Suppose the following is established: For any $f \in C_b(D)$, 
  \begin{equation}\label{eq:auxEqAff3} t \mapsto \int_{D \setminus
      \{|\xi| \leq 1\}} f(\xi) e^{\langle g_t,\xi\rangle} \mu^i(\d
    \xi) \quad \text{is continuous.}  \end{equation}
  Then for $k \in I \setminus \{i\}$ and any $h \in C_b(D)$, one defines $f:=\chi_k h$, notes that $f \in C_b(D)$ (since $\chi \in C_b(D)$) and $\mu^i$-integrable by \eqref{eq:admissLevyMes1} and concludes that
  \[ t \mapsto \int_{D \setminus \{0\} } h(\xi) \chi_k(\xi) \mu^i(t,\d
    \xi) \quad \text{is continuous},\] by \eqref{eq:auxEqAff2} and
  \eqref{eq:auxEqAff3}. Thus $\chi_k(\cdot) \mu^i(t,\cdot)$ is weakly
  continuous and the last strong admissibility condition follows
  analogously with $f:=\chi_k^2 h$ and \eqref{eq:admissLevyMes2}.

  To verify \eqref{eq:auxEqAff3}, note that \eqref{eq:expMoments} and
  \cite[Lemma~A.2]{duffie2003} yield that the function
  $G_0\colon \R^d \to \R$ defined via
  \begin{equation}\label{eq:auxEqAff17} G_0(u):= \int_{D \setminus
      \{|\xi| \leq 1\}} f(\xi) e^{\langle u,\xi\rangle} \mu^i(\d
    \xi) \end{equation}
  is analytic. In particular, composing it with the continuous function $y$ preserves continuity and hence \eqref{eq:auxEqAff3} holds.

  Finally, it remains to argue that $b(\cdot)$ and $\beta(\cdot)$ are
  continuous. To show this, for any $i \in I \cup \{0\}$,
  $k \in I \cup J$ one uses
  $\mu^i(D \setminus \{|\xi| \leq 1\}) < \infty$ (since $\chi$ is
  bounded away from $0$ on $D \setminus \{|\xi| \leq 1\}$ and by
  \eqref{eq:admissLevyMes1} and \eqref{eq:admissLevyMes2}) to
  decompose
  \begin{equation}\label{eq:auxEqAff15}\begin{aligned} \int_{D
        \setminus \{0\}} \chi_k(\xi)(e^{\langle
        g_t,\xi\rangle}-1)\mu^i(\d \xi) & = \int_{\{|\xi| \leq 1\}
        \setminus \{0\}} \chi_k(\xi)(e^{\langle g_t,\xi\rangle}-1)
      \mu^i(\d \xi) \\ & \quad + \int_{D \setminus \{|\xi| \leq 1\}}
      \chi_k(\xi) e^{\langle g_t,\xi\rangle} \mu^i(\d \xi) - \int_{D
        \setminus \{|\xi| \leq 1\}} \chi_k(\xi)\mu^i(\d
      \xi). \end{aligned} \end{equation} The second term is continuous
  in $t$ by \eqref{eq:auxEqAff3} and so it remains to show that the
  first integral is continuous in $t$. But this follows from dominated
  convergence: for any $T > 0$ one may use Lipschitz continuity of
  $\exp$, continuity of $g$, the Cauchy-Schwarz inequality and the
  properties of $\chi$ to find $C_0,C_1,C_2>0$ such that for all
  $t \in [0,T], \xi \in \{|\xi| \leq 1\} \setminus \{0\}$,
  \[ |\chi_k(\xi)(e^{\langle g_t,\xi\rangle}-1)| \leq C_0
    |\chi_k(\xi)| |\langle g_t,\xi\rangle|\leq C_1|\chi(\xi)|^2
    \frac{|\xi|}{|\chi(\xi)|} \leq C_2 |\chi(\xi)|^2. \] But
  \eqref{eq:admissLevyMes1} and \eqref{eq:admissLevyMes2} imply
  $\int_{\{|\xi| \leq 1\} \setminus \{0\}} |\chi(\xi)|^2 \mu^i(\d \xi)
  < \infty$ and thus the claim.

  \textbf{Verification of \eqref{eq:KallsenMuhleKarbeCond}:} Finally,
  again \eqref{eq:expMoments} and \cite[Lemma~A.2]{duffie2003} applied
  to the measure $|\xi|\mu^i(\d \xi)$ on $D \setminus \{|\xi|>1\}$
  (which is finite by \eqref{eq:expMoments}) shows that the function
  on $G\colon \R^d \to \R$ defined via
  \[ G(u) := \int_{D \setminus \{|\xi| \leq 1\}} |\xi| e^{\langle
      u,\xi\rangle} \mu^i(\d \xi) \] is analytic and thus for
  $i,k \in I$,
  \[\sup_{t \in [0,T]} \int_{\{\xi_k>1\}}\xi_k e^{g_t^\top \xi}
    \mu^i(\d \xi) \leq \sup_{t \in [0,T]} \int_{D \setminus
      \{|\xi|\leq 1\}} |\xi| e^{g_t^\top \xi} \mu^i(\d \xi) = \sup_{u
      \in K} G(u) < \infty,\] since $K:=\{g_s \, : \, s \in [0,T] \}$
  is a compact set by continuity of $g$.
\end{proof}

Based on Lemma~\ref{lem:admissible} and a result from
\cite{Kallsen2010} (alternatively, one could use \cite{cheridito2005})
we can now prove the following key tool:

\begin{proposition}\label{lem:truemartingale} Suppose
  $g \colon \R_+ \to \R^d$ is continuous and \eqref{eq:expMoments}
  holds. Then for any $x \in D$,
  \begin{itemize}
  \item[(i)] the process
    \begin{equation} \label{eq:Edef} E_t := \exp\left(\int_0^t
        g_u^\top \dd X_u -\int_0^t F(g_u) + \langle X_u , R(g_u)
        \rangle \dd u \right), \quad t \geq 0, \end{equation} is a
    $\P_x$-martingale,
  \item[(ii)] if for some $t \geq 0$ and all $s \geq 0$,
    $g(s)=g(s\wedge t)$, then $(E_{s \wedge t})_{s \geq 0}$ is the
    density process (w.r.t. $\P_x$) of a measure $\Q$ on
    $(\Omega,\mathcal{F})$ such that, under $\Q$, $X$ is a
    time-inhomogeneous affine process started from $(0,x)$ with
    parameters as in Lemma~\ref{lem:admissible}.
  \end{itemize}
\end{proposition}
\begin{proof}
  The proof of Proposition~\ref{lem:truemartingale} is structured as
  follows: In Step 1, $E$ in \eqref{eq:Edef} is rewritten as
  $\mathcal{E}(M)$ for a suitable local martingale $M$. In Step 2 it
  is verified that Lemma~\ref{lem:admissible} implies conditions
  \eqref{eq:KMKcond1}, \eqref{eq:KMKcond2} and \eqref{eq:KMKcond3}
  below. Finally, in Step 3 we combine Step 1 and 2 with
  \cite{Kallsen2010} and obtain (i) and (ii).

  \textbf{Step 1:} We follow the notation and definitions of
  \cite{Jacod2003}.

  Denote by $\mu^X$ the jump-measure and by $X^c$ the continuous
  martingale part of $X$, respectively. By \eqref{eq:expMoments},
  $|e^{g^\top x} -1-g^\top \chi(x)|*\nu$ is an adapted, continuous,
  increasing $\R$-valued process and thus (combining
  \cite[Lemma~I.3.10 and Proposition~II.1.28]{Jacod2003})
  $e^{g^\top x} -1+ g^\top \chi(x) \in G_{loc}(\mu^X)$. By linearity
  and \cite[Theorem~II.2.34]{Jacod2003},
  $g^\top \chi(x) \in G_{loc}(\mu^X)$ and so also
  $e^{g^\top x} -1 \in G_{loc}(\mu^X)$. Thus by
  \cite[Theorem~II.1.8(ii)]{Jacod2003}, the process
  \begin{equation}\label{eq:Mdef} M_t = \int_0^t g_s^\top \dd X^c_s +
    (e^{g^\top x} -1)* (\mu^X - \nu)_t, \quad t \geq 0 \end{equation}
  is a local martingale. By an argument as above and \cite[Corollary~II.2.38]{Jacod2003}, $g^\top x \in G_{loc}(\mu^X)$ and $W:= e^{g^\top x} - 1 - g^\top x  \in G_{loc}(\mu^X)$ and thus, using $\Delta M_t = e^{g_t^\top \Delta X_t} - 1 > -1$ one has
  \begin{equation}\label{eq:auxEqAff1}\begin{aligned}
      (\log(1+x)-x)*\mu^M & = (-g^\top x + e^{g^\top x} - 1)*\mu^X \\
      & = W*(\mu^X-\nu) + W*\nu \\ & \stackrel{\eqref{eq:Mdef}}{=}
      (-g^\top x) *(\mu^X-\nu) + M + \int_0^\cdot -g_s^\top \dd X^c_s
      + W*\nu \\ & = \int_0^\cdot -g_s^\top \dd X_s + M + W* \nu +
      \int_0^\cdot g_s^\top \beta(X_s) \dd s + g^\top (x-\chi(x))*\nu
      \\ & = - \int_0^\cdot g_s^\top \dd X_s + M + (e^{g^\top x} - 1 -
      g^\top \chi(x))*\nu + \int_0^\cdot g_s^\top \beta(X_s) \dd
      s. \end{aligned} \end{equation}

  Denoting by $\mathcal{E}$ the stochastic exponential, the definition
  (see also \cite[Theorem~8.10]{Jacod2003}) and
  \eqref{eq:characteristics} yields
  \begin{equation} \label{eq:auxEqAff5} \begin{aligned}
      \mathcal{E}(M)_t & = \exp\left(M_t-\frac{1}{2}\int_0^t g_s^\top
        \alpha(X_s) g_s \dd s - (\log(1+x)-x)*\mu^M_t \right)
      \\
      & \stackrel{\eqref{eq:auxEqAff1}}{=} \exp\bigg(\int_0^t g_u^\top \dd X_u - \int_0^t g_u^\top \beta(X_u)\dd u  -\frac{1}{2}\int_0^t g_s^\top \alpha(X_s) g_s \dd s \\ & \quad \quad + (g^\top \chi(x) -e^{g^\top x}+1)*\nu_t \bigg) \\
      & = \exp\left(\int_0^t g_u^\top \dd X_u -\int_0^t F(g_u) +
        \langle X_u , R(g_u) \rangle \dd u \right),
    \end{aligned}\end{equation}
  where the last step follows by definition \eqref{eq:FandR}.

  \textbf{Step 2:} Define $W\colon \R_+ \times \R^d \to [0,\infty)$ by
  $W(t,x):=e^{\langle g_t ,x \rangle}$. We now show that for all
  $j \in I \cup \{0\} $, $t \geq 0$,
  \begin{align} \label{eq:KMKcond1} \int_0^t \int_{D\setminus \{0\}}
    (1-\sqrt{W(s,x)})^2 \mu^j(\d x)\d s < \infty, \, &
    \\ \label{eq:KMKcond2} \int_{D\setminus
      \{0\}}|\chi(x)(W(t,x)-1)|\mu^j(\d x) < \infty, \, &
    \\ \label{eq:KMKcond3} \text{the measure }
    \chi_k(W(t,x)-1)(W(t,x)-1)\mu^j(\d x) & \text{ is weakly
      continuous in } t \in \R_+.
  \end{align}

  It remains to argue that \eqref{eq:KMKcond1}-\eqref{eq:KMKcond3} are
  indeed satisfied. Since $\exp$ is Lipschitz continuous and $g$ is
  continuous, there exists $C\geq 0$ such that for all $s \in [0,t]$,
  $|x|\leq 1$,
  \[ |1-\sqrt{W(s,x)}|=|1-e^{-\frac{1}{2}\langle g_s ,x \rangle}|\leq
    C |\langle g_s ,x \rangle|\leq C|g_s||x|. \] Taking
  $K \subset \R^d$ compact with $g_s, \frac{1}{2} g_s \in K$ for all
  $s \in [0,t]$ and splitting the integral in $\{|x|\leq 1\}$ and
  $\{|x| \geq 1\}$, we obtain (for $G_0$ as in \eqref{eq:auxEqAff17}
  with $f=1$)
  \[\begin{aligned} \int_0^t \int_{D\setminus \{0\}} (1-e^{-\frac{1}{2}\langle g_s ,x \rangle})^2 \mu^j(\d x)\d s \leq C & \int_0^t |g_s|^2 \dd s \int_{|x|\leq 1}|x|^2 \mu^j(\d x) \\
      & + 2 \sup_{u \in K} G_0(u) + \int_{D\setminus \{|x|\leq 1\}}
      \mu^j(\d x), \end{aligned}\] which is finite by the
  integrability properties of the L\'evy-measures
  \eqref{eq:admissLevyMes1}, \eqref{eq:admissLevyMes2} and since $G_0$
  is continuous.  Thus \eqref{eq:KMKcond1} indeed holds and an
  analogous reasoning gives \eqref{eq:KMKcond2}.

  To establish \eqref{eq:KMKcond3}, denote
  $\tilde{\mu}(t,\d x):=\chi_k(W(t,x)-1)(W(t,x)-1)\mu^j(\d x)$ and
  again consider $D \setminus \{|\xi|\leq 1\}$ and
  $\{|\xi| \leq 1\} \setminus \{0\}$ separately, i.e. for
  $f \in C_b(D)$ write
  \begin{equation}\label{eq:auxEqAff16} \int_{D\setminus \{0\}} f(\xi)
    \tilde{\mu}(t,\d \xi) = \int_{D \setminus \{|\xi| \leq
      1\}}f(\xi)\tilde{\mu}(t,\d \xi) + \int_{ \{|\xi| \leq 1\}
      \setminus \{0\}}f(\xi)\tilde{\mu}(t,\d \xi). \end{equation}
  The second term is continuous in $t$ by dominated convergence and the same argument used to show that $b$ and $\beta$ are continuous. The first term in  \eqref{eq:auxEqAff16} is the composition of $F_0 \colon \R^d \to \R_+$ defined by
  \[F_0(u):= \int_{D \setminus \{|\xi| \leq
      1\}}f(\xi)\chi_k(e^{\langle u ,\xi \rangle}-1)(e^{\langle u ,\xi
      \rangle}-1)\mu^j(\d \xi) \] and $g$. To establish
  \eqref{eq:KMKcond3} it thus suffices to show that $F_0$ is
  continuous. To see this, assume $f \geq 0$ (for general $f$ apply
  the subsequent argument to the positive and negative parts of $f$
  separately), define $h \colon [-1,\infty) \to \R$ by
  $h(z):=z^2-z\chi_k(z)$ and write
  \[ F_0(u)=G_0(2u)-2 G_0(u) + G_0(0) - \int_{D \setminus \{|\xi| \leq
      1\}}f(\xi)h(e^{\langle u ,\xi \rangle}-1)\mu^j(\d \xi) \] with
  $G_0$ as in \eqref{eq:auxEqAff17}. For the truncation function
  $\chi$ chosen in \cite{Kallsen2010}, $h(z)=\max(0,z^2-z)$ for all
  $z \in [-1,\infty)$ and so $h$ is non-decreasing and convex. In
  particular for any $\xi \in D \setminus \{|\xi| \leq 1\}$, the
  function on $\R^d$ defined by
  $u \mapsto h(e^{\langle u ,\xi \rangle}-1)$ is convex and so
  \[ u \mapsto \int_{D \setminus \{|\xi| \leq 1\}}f(\xi)h(e^{\langle u
      ,\xi \rangle}-1)\mu^j(\d \xi) \] is a ($\R_+$-valued) convex
  function on $\R^d$. \cite[Corollary~10.1.1]{Rockafellar1970} implies
  that it is continuous and so the proof is complete.

  \textbf{Step 3:} Recall that \eqref{eq:collectionOfParameters} with
  $c=0$ and $\gamma =0$ is strongly admissible in the sense of
  \cite[Definition~2.4]{Kallsen2010} and by Lemma~\ref{lem:admissible}
  the same holds for \eqref{eq:parametersFilter}. Furthermore, recall
  the definition of $M$ in \eqref{eq:Mdef}.  Since
  $g \colon \R_+ \to \R^d$ and
  $W\colon \R_+ \times \R^d \to [0,\infty)$ (defined above) are
  continuous, satisfy (by Step 2) conditions \eqref{eq:KMKcond1},
  \eqref{eq:KMKcond2} and \eqref{eq:KMKcond3} and since
  \eqref{eq:KallsenMuhleKarbeCond} holds,
  \cite[Theorem~4.1]{Kallsen2010} and its proof show that
  $\mathcal{E}(M)$ is a martingale and that $\mathcal{E}(M)$ can be
  used as the density process of a probability measure $\Q$ that is
  locally absolutely continuous w.r.t. $\P_x$ and has the properties
  stated in (ii). But $E_t=\mathcal{E}(M)_t$ (by Step 1) and hence the
  claim.
\end{proof}

\subsubsection{Estimates for $R$}

\begin{lemma}\label{lem:auxLem1} There exists a function
  $g \in C(\R^d,\R_+)$ such that $g(x) = g((x_I^+,x_J))$ for all
  $x \in \R^d$ (with $x_I^+ = (x_1^+,\ldots,x_m^+)$) and for any
  $u \in \C^d$,
  \begin{equation}\label{eq:MKRineq} \Re \langle \bar u_I,
    R_I(u)\rangle \leq g(\Re u)(1+|u_J|^2)(1+|u_I|^2).\end{equation}
\end{lemma}
\begin{proof}
  Inequality \eqref{eq:MKRineq} is derived in
  \cite[Lemma~5.5]{keller-ressel2015} with
  \begin{equation}\label{eq:auxEqAff13} g(x):= c_0 (1+x_I^+) + c_1
    e^{x^+} + \sum_{i=1}^m \int_{D \cap |\xi|\geq 1} e^{\langle \xi, x
      \rangle} \mu^i(\d \xi) + \int_{D \cap |\xi|\leq 1}
    \xi_i(e^{\xi_i x_i^+}-1)\mu^i(\d \xi) \end{equation}
  for some $c_0, c_1 > 0$. Since $\xi_k \geq 0$ for all $k \in I$, $e^{\langle \xi , x \rangle} \leq e^{\langle \xi , (x_I^+,x_J) \rangle}$ and so \eqref{eq:MKRineq} remains valid if instead of $g$ one uses $g((x_I^+,x_J))$. Continuity of this function follows by \eqref{eq:expMoments} and so the lemma is proved.
\end{proof}

\begin{lemma}\label{lem:auxLem2}
  Let $r > 0$ and
  $S_r:=\{ u \in \C^d \; : \; \forall \in I\; \Re u_i \leq r, |u_J|
  \leq r \}$.  Then there exists $C>0$ such that for all $u \in S_r$,
  \[ | R_I(u)| \leq C(1+|u|^2). \]
\end{lemma}
\begin{proof}
  By the triangle inequality it suffices to find for each $i \in I$ a
  constant $C_i > 0$ such that $|R_i(u)|\leq C_i(1+|u|^2)$ for all
  $u \in S_r$. For $i \in I$,
  \[|R_i(u)|\leq
    \left(\frac{1}{2}|\alpha^i|+|\beta^i|\right)(1+|u|^2)+\int_{D
      \setminus \{0\}}\left|e^{\langle \xi, u \rangle } - 1 - \langle
      \chi(\xi), u \rangle \right| \mu^i(\d \xi) \] and so we only
  need to analyze the $\mu^i$-integral. Set
  $B:=\{z \in \C \; : \; \Re z \leq (d+1)r \}$, then for all
  $z \in B$,
  \[ |\exp(z)-1-z| \leq |z| \sup_{t \in (0,1)}|e^{t z}-1| = |z|
    \sup_{t \in (0,1)}|z \int_0^t e^{ s z} \dd s| \leq |z|^2
    e^{(d+1)r}.  \] Furthermore, for any $\xi \in D$ with
  $|\xi|\leq 1$ and $u \in S_r$,
  \[ \Re \langle \xi , u \rangle = \langle \xi , \Re u \rangle \leq
    \left(\sum_{i \in I} \xi_i + |\xi_J|\right) r \leq (d+1) r \]
  implies $\langle \xi, u \rangle \in B$. Combining these two
  observations with the Cauchy-Schwarz inequality and $\chi(\xi)=\xi$
  on $\{|\xi| \leq 1\}$ one obtains
  \[ \int_{D \cap \{|\xi| \leq 1\}} \left|e^{\langle \xi, u \rangle }
      - 1 - \langle \chi(\xi), u \rangle \right| \mu^i(\d \xi) \leq
    |u|^2 e^{(d+1)r} \int_{D \cap \{|\xi| \leq 1\}} |\xi|^2 \mu^i(\d
    \xi) =: |u|^2 C_i^{0}, \] where $C_i^0$ is finite because of
  \eqref{eq:admissLevyMes1} and \eqref{eq:admissLevyMes2}.

  By \eqref{eq:expMoments} and \cite[Lemma~A.2]{duffie2003}, the
  function
  \[ \tilde{R}_i(u):= \int_{D \setminus \{|\xi| \leq 1\}} e^{\langle
      \xi, u \rangle } \mu^i(\d \xi) , \quad u \in \C^d \] is
  analytic. In particular, $C_0:=\sup_{|u|\leq r} |\tilde{R}_i(u)|$ is
  finite.

  Since $\chi$ is bounded away from $0$ on
  $D \setminus \{|\xi| \leq 1\}$,
  $C:=\mu^i(D \setminus \{|\xi| \leq 1\})$ is finite (by
  \eqref{eq:admissLevyMes1} and \eqref{eq:admissLevyMes2}). Combining
  this with $|\chi(\xi)|\leq d$ one obtains for $u \in S_r$
  \[\begin{aligned} \int_{D \setminus \{|\xi| \leq 1\}}
      \left|e^{\langle \xi, u \rangle } - 1 - \langle \chi(\xi), u
        \rangle \right| \mu^i(\d \xi) & \leq \int_{D \setminus \{|\xi|
        \leq 1\}} e^{\langle \xi, \Re u \rangle } \mu^i(\d \xi) + C(1
      + d|u|) \\ & \leq C_0 + C(1 + |u|) ,
    \end{aligned}\] where we have used $\xi_I \in \R_+^m$ for the last
  estimate. Combining all the estimates yields the desired statement.
\end{proof}

\subsubsection{Properties of $\bar{T}^y$} To prepare the proof of
Proposition~\ref{thm:relationToDavis} we provide two additional
Lemmas. The first is an application of Itô's lemma and essentially
identifies the extended generator of $\bar{T}^y$ in
\eqref{eq:modifiedZakaiFunctional}. The second Lemma rephrases a
result from \cite{Filipovic2005}.

Recall that $H(x) = (x_0)_I^\top x_I + c_0$ and for
$f \in C_0^{1,2}(\R_+\times D)$, $t \geq 0$, set
\begin{equation}\label{eq:Aydef} \begin{aligned}
    \A^y_t f(t,x) & = \A f(t,x) + f(t,x)[F(-y_t) + \langle x, R(-y_t)
    \rangle - H(x)] \\ & - \langle \alpha(x)y_t, \nabla_x f(t,x)
    \rangle + \int_{D \setminus \{0\}} [f(t,x+\xi) -
    f(t,x)](e^{-\langle y_t,\xi\rangle}-1) K(x,\d \xi). \end{aligned}
\end{equation}

\begin{proposition}\label{prop:extendedGenerator}
  Suppose $y \colon \R_+ \to \R^d$ is continuous,
  \eqref{eq:collectionOfParameters} are admissible with $c=0$,
  $\gamma = 0$ and \eqref{eq:expMoments} holds.  Define
  \begin{equation}\label{eq:Udef} U_t := \exp\left(-\int_0^t y_u^\top
      \dd X_u - \int_0^t H(X_u) \dd u \right) \end{equation}
  and $\A^y$ as in \eqref{eq:Aydef}.
  Then for any $f \in C_0^{1,2}(\R_+ \times D)$, the process 
  \begin{equation}\label{eq:martingaleProblemEt} U_t f(t,X_t) -
    f(0,X_0) - \int_0^t U_u (\partial_u + \A^y_u ) f(u, X_u) \dd u ,
    \quad t \geq 0, \end{equation}
  is a local martingale.

\end{proposition}

\begin{proof}
  Define $E$ by \eqref{eq:Edef} with $g:=-y$. Then
  $E_t = \mathcal{E}(M)_t$ for $M$ as in \eqref{eq:Mdef}. Furthermore,
  $U_t = E_t \exp(V_t)$, where
  \[ V_t = \int_0^t F(-y_u) + \langle X_u , R(-y_u) \rangle \dd u -
    \int_0^t H(X_u) \dd u \] is continuous and of bounded variation,
  $[E,V]=0$ and thus
  \[ d U_t = d(\mathcal{E}(M)_t\exp(V_t)) = U_{t-} \dd M_t + U_t \dd
    V_t. \]

  For $f \in C_0^{1,2}(\R_+ \times D)$, Itô's formula shows that
  \[ f(t,X_t) = f(0,X_0) + \int_0^t (\partial_s + \A) f(s,X_s) \dd s +
    N_t, \quad t \geq 0,\] where $N$ is a local martingale and the
  continuous part of $N$ is
  $ \int \nabla_x f(s,X_{s-})^\top \dd X_s^c$. Combining this with
  \eqref{eq:characteristics}, the definition \eqref{eq:Mdef}, the fact
  that $f$ is bounded and $e^{-y^\top x} -1 \in G_{loc}(\mu^X)$, we
  obtain
  \[\begin{aligned} [M,N]_t & = \langle M^c , N^c \rangle_t + \sum_{s \leq t}(f(s,X_s)-f(s,X_{s-}))(e^{-y_s\Delta X_s} - 1) \\
      & \stackrel{.}{=} - \int_0^t y_s^\top \alpha(X_s) \nabla_x
      f(s,X_{s}) \dd s \\ & + \int_0^t
      (f(s,X_{s-}+\xi)-f(s,X_{s-}))(e^{-y_u^\top \xi} - 1)K(X_{u-},\dd
      \xi) \dd u , \end{aligned}\] where $U \stackrel{.}{=} V$ means
  that $U - V$ is a local martingale.

  Putting everything together, Itô's formula written in differential
  form gives
  \[\begin{aligned} d U_t f(t,X_t) \stackrel{.}{=} &  U_{t} (\partial_t + \A) f(t,X_t)\dd t  + f(t,X_{t-}) U_t d V_t + d [U,N]_t \\
      \stackrel{.}{=} & U_{t} \A^y_t f(t,X_t)\dd t ,
    \end{aligned}\] which shows that \eqref{eq:martingaleProblemEt} is
  a local martingale.\end{proof}

In the following Lemma, we allow the function spaces (defined before)
to contain complex valued functions.
\begin{lemma} \label{lem:spacetimeharmonic} There exists a dense
  subset $L \subset C_0(D)$ with the following property: for any
  $T > 0$, $h \in L$ there exists $u \in C^{1,2}([0,T] \times D)$
  bounded, satisfying
  \begin{equation}\label{eq:spacetimeharmonic}\begin{aligned}
      \partial_t u(t,x) + \A_t^y u(t,x) & = 0 \quad & (t,x) \in [0,T)\times D, \\
      u(T,x) & = h(x) \quad & x \in D.
    \end{aligned}\end{equation}

\end{lemma}

\begin{proof} Denote by $\Theta_0 \subset C_0(D)$ the set of
  $\C$-valued functions from \cite[Proposition~8.2]{duffie2003}. Any
  $h \in \Theta_0$ is of the form
  \[ h(x) = \int_{\R^n} e^{(v,i q)^\top x} g(q) \dd q , \quad x \in
    D, \] for some $g \in C_c^\infty(\R^n)$ and $v \in \C_{--}^m$.
  Denote by $L$ the complex linear span of $\Theta_0$.  In
  \cite[Lemma~8.4]{duffie2003} it is shown that $L$ dense in $C_0(D)$.

  Fix $T > 0$ and $h \in \Theta_0$.  For $(t,x)\in [0,T]\times D$,
  define $u(t,x):=P_{t,T}h(x)$. Then $u$ is bounded, satisfies
  $u(T,\cdot)=h$ and, as established in the proof of
  \cite[Proposition~6.3]{Filipovic2005},
  $u \in C^{1,2}([0,T] \times D)$ and \eqref{eq:spacetimeharmonic}
  indeed holds.
\end{proof}

\subsection{Proof of Proposition~\ref{thm:relationToDavis}}

\begin{proof}[Proof of Proposition~\ref{thm:relationToDavis}]
  Since $x_0^{i} > 0$ for $i \in I$, $y_0=0$, $R(0)=0$ and $R$, $y$
  are continuous, there exists $T>0$ such that
  $x_0^{i} - R_i(-y_t) \geq 0$ for $t \in [0,T]$. Taking
  $c_0=\sup_{t \in [0,T]} F(-y_t)$ it follows that $c(t) \geq 0$ and
  $\gamma(t) \in \R^m_+$ for all $t \in [0,T]$. Combining this with
  Lemma~\ref{lem:admissible}, it follows that the parameters are
  indeed admissible.

  To prove the proposition, it suffices to show
  $\bar{T}_t^y = P_{0,t}^y$ and for this it is sufficient to show
  $\bar{T}_t^y h = P_{0,t}^y h$ for all $h$ in a dense subset of
  $C_0(D)$. Taking $L$ from Lemma~\ref{lem:spacetimeharmonic}, for any
  $h \in L$ we find $f \in C_0^{1,2}([0,t]\times D)$ such that
  $f(t,\cdot) = h$ and the $\d u$-integral in
  \eqref{eq:martingaleProblemEt} vanishes. Hence
  \[ N_s := U_{s\wedge t} f(s\wedge t,X_{s\wedge t}), \quad s \geq
    0,\] is a local martingale by
  Proposition~\ref{prop:extendedGenerator}, where $U$ is as in
  \eqref{eq:Udef}.
  On the other hand,
  \[U_t = \exp(V_t)E_t, \] where $E$ is defined in \eqref{eq:Edef}
  (with $g=-y$) and
  \[V_t := \int_0^t F(-y_u) + \langle X_u , R(-y_u) \rangle \dd u -
    \int_0^t H(X_u) \dd u ,\quad t \geq 0.\] Since $\beta^j = 0$ for
  $j \in J$ and $x^i \geq 0$ for $i \in I$,
  \[ H(x) - F(-y_t) - \langle x , R(-y_t) \rangle = c(t) + x^
    \top\gamma(t) \geq 0 \] for all $(t,x) \in [0,T] \times D$. Thus
  $V_t \leq 0$ for $t \in [0,T]$ and $\exp(V_t)$ is bounded on
  $[0,T]$.  Since $E$ is a martingale by
  Proposition~\ref{lem:truemartingale}, the local martingale $N$
  satisfies
  \[N_s = \exp(V_{t \wedge s})E_{t \wedge s} f(s\wedge t, X_{s \wedge
      t}), \quad s \geq 0, \] and is the product of a bounded process
  and a martingale. Thus $N$ is a true martingale and combining this
  with $f(t,\cdot)=h$, the definition
  \eqref{eq:modifiedZakaiFunctional} and $f(s,\cdot)= P_{s,t} h$ (see
  Lemma~\ref{lem:spacetimeharmonic}) yields
  \begin{equation} \begin{aligned} \bar{T}_t^y h(x) = \E_x[U_t
      f(t,X_t)] = \E_x[N_t] = \E_x[N_0] = f(0,x) =
      P_{0,t}h(x). \end{aligned} \end{equation}

\end{proof}

\subsection{Proof of Theorem~\ref{thm:genRiccatiFilter} and
  \ref{thm:timeInhomogeneousAffine}}\label{sec:mainproofs}

\begin{proof}[Proof of Theorem~\ref{thm:genRiccatiFilter}]
  We proceed in two steps: First \eqref{eq:FourierCoeff} is verified
  under the assumption that a solution to \eqref{eq:genRiccatiFilter}
  exists. In the second part, existence and uniqueness for
  \eqref{eq:genRiccatiFilter} is established.

  \textbf{Expression for the Fourier coefficients:} Since $\Psi$ is
  continuously differentiable, each component is of finite variation
  and thus $[\Psi^j,X^j]=0$ for all $j$.  By the product rule and
  \eqref{eq:genRiccatiFilter},
  \begin{equation}\label{eq:auxEqAff4}\begin{aligned} (u+y_T)^\top X_T
      - \Psi(0,T,u)^\top X_0 & =
      \Psi(T,T,u)^\top X_T - \Psi(0,T,u)^\top X_0\\  & = \int_0^T \Psi(s,T,u)^\top \dd X_s  + \int_0^T X_s^\top \partial_s \Psi(s,T,u) \dd s \\
      &= \int_0^T \Psi(s,T,u)^\top \dd X_s - \int_0^T X_s^\top(
      R(\Psi(s,T,u)-y_s) - \gamma_s) \dd s.
    \end{aligned}\end{equation}

  By Proposition~\ref{lem:truemartingale} applied to the continuous
  function $g \colon \R_+ \to \R^d$ defined by
  \[ g(t):= \Psi(t \wedge T ,T,u)-y_{t \wedge T},\] the process
  \[\tilde{E}_t := \exp\left(\int_0^t g_u^\top \dd X_u -\int_0^t
      F(g_u) + \langle X_u , R(g_u) \rangle \dd u \right), \quad t
    \geq 0, \] is a martingale.

  Combining this with \eqref{eq:auxEqAff4} and the definition of
  $\rho$ we obtain
  \begin{equation}\label{eq:Fourierderivation} \begin{aligned}
      \rho^x_T(f_u,y) & = \E\left[\exp \left((u+y_T)^\top X_T - \int_0^T y_s^\top \dd X_s- \int_0^T c_s + X_s^\top \gamma_s \dd s  \right)\right]\\
      & = \E\bigg[ \exp \bigg(\Psi(0,T,u)^\top X_0 + \int_0^T (\Psi(s,T,u) - y_s)^\top \dd X_s   \\ & \quad \quad - \int_0^T c_s  X_s^\top R(\Psi(s,T,u)-y_s) \dd s  \bigg)\bigg] \\
      & = \E\left[ \tilde{E}_T \exp\left(\Psi(0,T,u)^\top X_0 +
          \int_0^T F(\Psi(s,T,u) - y_s) - c_s \dd s \right)\right]
      \\
      & = \exp(\Phi(0,T,u)+ \Psi(0,T,u)^\top x ).
    \end{aligned}
  \end{equation}

  \textbf{Existence and uniqueness of solutions to
    \eqref{eq:genRiccatiFilter}:} Suppose first for some $T>0$ there
  exists $\tilde{\Psi} \in C^1([0,T],\R^d)$ satisfying
  \begin{equation}\label{eq:forwardRicatti} \partial_t
    \tilde{\Psi}(t,u)=R(\tilde{\Psi}(t,u)+y_T-y_{T-t})-\gamma_{T-t},
    \quad \tilde{\Psi}(0,u)=u. \end{equation}
  Then a solution to \eqref{eq:genRiccatiFilter} is obtained by setting $\Psi(t,T,u):=\tilde{\Psi}(T-t,u)+y_T$ and 
  \[ \Phi(t,T,u) = \int_t^T F(\Psi(s,T,u)-y_s) - c_s \dd s, \quad t
    \in [0,T]. \] Conversely, any solution to
  \eqref{eq:genRiccatiFilter} gives rise to $\tilde{\Psi}$ satisfying
  \eqref{eq:forwardRicatti} by setting
  $\tilde{\Psi}(t,u):=\Psi(T-t,T,u)-y_T$.  Thus, to prove the theorem
  it suffices to construct $T>0$ such that for all $u \in i \R^d$
  there exists a unique $\tilde{\Psi}(\cdot,u) \in C^1([0,T],\R^d)$
  satisfying \eqref{eq:forwardRicatti}. To do so, we will establish
  the following statements:

  \begin{itemize}
  \item[(i)] for any $T>0$, $u \in \C^d$, there exists
    $t_+(u,T) \in (0,\infty]$ such that \eqref{eq:forwardRicatti} has
    a unique solution on $[0,t_+(u,T))$. If $t_+(u,T) < \infty$, then
    $\lim_{t \uparrow t_+(u,T)} |\tilde{\Psi}(t,u)|=\infty$.
  \item[(ii)] there exists $T_0 > 0$ such that $t_+(0,T_0) > T_0$,
    i.e. the solution to \eqref{eq:forwardRicatti} with $u = 0$,
    $T=T_0$ exists on $[0,T_0]$.
  \item[(iii)] for any $u \in i \R^d$, $t_+(u,T_0) > T_0$.
  \end{itemize}
  Then (iii) implies that for any $u \in i \R^d$ there exists a unique
  solution to \eqref{eq:forwardRicatti} on $[0,T_0]$, which proves the
  theorem. We now show (i)-(iii). In what follows, we set $y_r:=0$ for
  $r<0$ so that $y \in C(\R,\R^d)$.

  \textbf{(i)} By \cite[Lemma~5.3]{duffie2003} and
  \eqref{eq:expMoments}, $R$ is an analytic function. In particular it
  is locally Lipschitz continuous. Combining this with the fact that
  $y$ is continuous, (i) follows from the global existence and
  uniqueness result for ordinary differential equations
  \cite[Theorem~7.6]{Amann1990}.

  \textbf{(ii)} For $(t,z,T) \in \R \times \C^d \times \R$, set
  \[ f(t,z,T):= R(z+y_T - y_{T-t}) - \gamma_{T-t}. \] Then
  $f \in C(\R \times \C^d \times \R, \C^d)$ and, since $R$ is locally
  Lipschitz-continuous, the prerequisites
  of\cite[Theorem~8.3]{Amann1990} are satisfied. Thus, the set
  \[ D:=\{(t,\tau,u,T)\in \R\times\R \times \C^d \times \R \; : \; t
    \in J(\tau,u,T) \} \] is open, where $J(\tau,u,T)$ is the maximal
  interval of existence of the (unique) solution to
  \[ \dot x(t) = f(t,x(t),T), \quad x(\tau)=u. \] Since
  $(0,0,0,0) \in D$ and $D$ is open, $(T_0,0,0,T_0) \in D$ for $T_0>0$
  small enough. Thus $T_0 \in J(0,0,T_0)$ and, since the right
  endpoint of the open interval $J(0,0,T_0)$ is $t_+(0,T_0)$, the
  claim follows.

  \textbf{(iii)} Fix $u \in i \R^d$. By (ii), $t_+(0,T_0)>T_0$ and so
  it suffices to show that $t_+(u,T_0) \geq t_+(0,T_0)$ or, by (i),
  that $|\tilde{\Psi}(t,u)|$ does not explode on $[0,T_0]$. Consider
  the $J$-components first. By \eqref{eq:driftcond3}, for $j \in J$
  \eqref{eq:forwardRicatti} is given as
  \begin{equation}\label{eq:auxEqAff11} \partial_t
    \tilde{\Psi}_j(t,u)= \langle \beta^j ,
    \tilde{\Psi}_J(t,u)+y_J(T)-y_J(T-t) \rangle -(\gamma_{T-t})_j ,
    \quad \tilde{\Psi}_j(0,u)=u_j \end{equation}
  and, as this is a system of first order linear equations, $\tilde{\Psi}_J(t,u)$ exists for all $t \geq 0$. Thus it remains to analyze the $I$-components. We claim that there exists constants $c_0, c_1 > 0$ such that for all $t \in [0,T_0 \wedge t_+(u,T_0))$
  \begin{equation}
    \label{eq:nonexplosion}
    \partial_t |\tilde{\Psi}_I(t,u)|^2 \leq c_0(c_1+|\tilde{\Psi}_I(t,u)|^2).
  \end{equation}
  Assuming that \eqref{eq:nonexplosion} has been established,
  Gronwall's inequality applied to $c_1+|\tilde{\Psi}_I(t,u)|^2$
  implies
  \begin{equation}\label{eq:auxEqAff7} |\tilde{\Psi}_I(t,u)|^2 \leq
    (c_1+|u_I|^2)\exp\left(c_0 t\right) - c_1 \end{equation}
  for all $t \in [0,T_0 \wedge t_+(u,T_0))$. This allows to conclude (iii) by contradiction: If $T_0 \geq t_+(u,T_0)$, then \eqref{eq:auxEqAff7} holds for all $t \in [0,t_+(u,T_0))$ and the left hand side of \eqref{eq:auxEqAff7} explodes as $t \uparrow T_0$, whereas the right hand side is bounded by its value at $T_0$. Hence, by contradiction $T_0 < t_+(u,T_0)$ as claimed.

  Therefore it suffices to establish \eqref{eq:nonexplosion}.  To do
  so, we follow the proof of \cite[Proposition~6.1]{duffie2003} and
  \cite[Proposition~5.1]{keller-ressel2015}. For $t \in \R$, set
  $\bar{y}(t):= y_{T_0}-y_{T_0-t}$. As argued above,
  \eqref{eq:auxEqAff11} implies that $\tilde{\Psi}_J(t,u)$ exists for
  all $t \geq 0$. Furthermore, the real part of \eqref{eq:auxEqAff11}
  does not depend on $u$ and therefore
  \begin{equation}\label{eq:auxEqAff12}
    \Re \tilde{\Psi}_J(t,u) = \Re \tilde{\Psi}_J(t,0) = \tilde{\Psi}_J(t,0).
  \end{equation}
  Set $T:=t_+(0,T_0) \wedge t_+(u,T_0)$ and for
  $(t,x) \in [0,T] \times \R^m$,
  \[ f(t,x):= R_I((x,\tilde{\Psi}_J(t,0)) +\bar{y}(t)) -
    (\gamma_{T_0-t})_I. \] Then by
  \cite[Lemma~5.7]{keller-ressel2015}, continuity of $y$ and Lipschitz
  continuity of $R_I$, $f$ satisfies the conditions of the comparison
  result \cite[Proposition~A.2]{Mayerhofer2011a}. Furthermore,
  \eqref{eq:auxEqAff12} and the inequality
  $\Re R_i(z) \leq R_i(\Re(z))$ (valid for all $z \in \C^d$) yield
  \[\partial_t \Re \tilde{\Psi}_i(t,u) - f_i(t,\Re
    \tilde{\Psi}_I(t,u)) \leq 0 = \partial_t \tilde{\Psi}_i(t,0) -
    f_i(t, \tilde{\Psi}_I(t,0)) , \quad \Re
    \tilde{\Psi}_i(0,u)=\tilde{\Psi}_i(0,0) \] for $t \in [0,T)$,
  $i \in I$. Hence the comparison result
  \cite[Proposition~A.2]{Mayerhofer2011a} implies
  \begin{equation}\label{eq:comparison}
    \Re \tilde{\Psi}_i(t,u) \leq \tilde{\Psi}_i(t,0), \quad \forall i \in I, \; t \in [0,t_+(0,T_0) \wedge t_+(u,T_0))
  \end{equation}

  For $t \in [0,T_0 \wedge t_+(u,T_0))$ one uses
  \eqref{eq:forwardRicatti} to write
  \begin{equation}\label{eq:auxEqAff6}\begin{aligned} \frac{1}{2}\partial_t |\tilde{\Psi}_I(t,u)|^2 & = \Re \langle \overline{\tilde{\Psi}_I}(t,u) ,  \partial_t \tilde{\Psi}_I(t,u) \rangle \\
      & = \Re \langle \overline{\tilde{\Psi}_I}(t,u)+ \bar{y}(t),
      R_I(\tilde{\Psi}(t,u)+\bar{y}(t))\rangle - \langle \bar{y}(t),
      \Re R_I(\tilde{\Psi}(t,u)+\bar{y}(t))\rangle \\ & \quad -
      \langle \Re \tilde{\Psi}_I(t,u),(\gamma_{T_0-t})_I \rangle \\ &
      = I_1 - I_2 - I_3, \end{aligned} \end{equation} where each $I_i$
  denotes an inner product. The three inner products in
  \eqref{eq:auxEqAff6} can be estimated separately:

  For the first one, denote by $g \in C(\R^d,\R_+)$ the function from
  Lemma~\ref{lem:auxLem1}, write $x^{+,I}:=(x_I^+,x_J)$ for
  $x \in \R^d$ and recall $g(x)= g(x^{+,I})$. By \eqref{eq:comparison}
  and the fact that $\tilde{\Psi}_J(t,u)$ exists for all $t \geq 0$,
  there exists $K \subset \R^d$ compact such that
  $(\Re \tilde{\Psi}(t,u)+\bar{y}(t))^{+,I} \in K$ for all
  $t \in [0,T_0 \wedge t_+(u,T_0))$. Hence Lemma~\ref{lem:auxLem1}
  yields
  \begin{equation} \label{eq:auxEqAff8} \begin{aligned}
      I_1 &\leq g(\Re \tilde{\Psi}(t,u)+\bar{y}(t))(1+|\tilde{\Psi}_J(t,u)+ \bar{y}_J(t)|^2) (1+ |\tilde{\Psi}_I(t,u)+ \bar{y}_I(t)|^2) \\
      & \leq 4 g((\Re \tilde{\Psi}(t,u)+\bar{y}(t))^{+,I})(1+|\tilde{\Psi}_J(t,u)+ \bar{y}_J(t)|^2) (C_1+ |\tilde{\Psi}_I(t,u)|^2) \\
      & \leq C_0 (C_1+ |\tilde{\Psi}_I(t,u)|^2)
    \end{aligned} \end{equation} where
  $C_0:=(4\sup_{x \in K} g(x) \wedge 1) \sup_{t \in [0,T_0]}
  (1+|\tilde{\Psi}_J(t,u)+ \bar{y}_J(t)|^2)$ and
  $C_1:=1+2 \sup_{t \in [0,T_0]}|\bar{y}(t)|^2$.

  For the second one, Lemma~\ref{lem:auxLem2}, the fact that
  $\tilde{\Psi}_J(t,u)$ exists for all $t \geq 0$ and
  \eqref{eq:comparison} yield that there exists $C > 0$ such that
  \begin{equation}\label{eq:auxEqAff9}\begin{aligned} - I_2 & \leq
      |\bar y(t)||R_I(\tilde{\Psi}(t,u)+\bar{y}(t))| \\ & \leq C C_1
      (1+ |\tilde{\Psi}(t,u)+ \bar{y}(t)|^2) \\ & \leq 4 C C_1 (C_0 +
      C_1 + |\tilde{\Psi}_I(t,u)|^2) \end{aligned} \end{equation} and
  for the last one
  \begin{equation}\label{eq:auxEqAff10} - I_3 \leq
    |\gamma_{T_0-t}||\tilde{\Psi}_I(t,u)| \leq \sup_{s \in
      [0,T_0]}|\gamma_s| (1+|\tilde{\Psi}_I(t,u)|^2).  \end{equation}
  Combining \eqref{eq:auxEqAff6} with the estimates \eqref{eq:auxEqAff8},\eqref{eq:auxEqAff9} and \eqref{eq:auxEqAff10} yields \eqref{eq:nonexplosion}, as desired.
\end{proof}

\begin{proof}[Proof of Theorem~\ref{thm:timeInhomogeneousAffine}]
  Precisely as in the derivation of \eqref{eq:Fourierderivation}, one
  combines the definition \eqref{eq:modifiedZakaiSmoother}, the
  product rule \eqref{eq:auxEqAff4} for $u=0$ and the definition of
  $E_t$ in \eqref{eq:Edef} to write
  \[\begin{aligned} \rho_t(G,y) & =\E \left[G(X_{[0,t]}) E_t  \exp\left(\Psi(0,t,0)^\top X_0 + \int_0^t F(g_s) - c_s \dd s \right) \right] \\
      & = \exp(\Phi(0,t,0)) \int_D \exp(\langle x, \Psi(0,t,0)
      \rangle) \E_{x}[G(X_{[0,t]}) E_t ] \pi_0(\d x).
    \end{aligned} \] But for any $x \in D$,
  \begin{equation*}
    \E_x \left[G(X_{[0,t]}) E_t \right]= \E_{\mathbb{Q}_{x}^{y,t}}\left[ G(X_{[0,t]})  \right]
  \end{equation*}
  with $\mathbb{Q}_{x}^{y,t} = \Q$ as in
  Lemma~\ref{lem:truemartingale}(ii). Thus the statement follows from
  the definition of $\bar{\pi}_t(G,y)$ and
  Lemma~\ref{lem:truemartingale}(ii).
\end{proof}

\section{Illustration: Filtering a Cox-Ingersoll-Ross process}\label{sec:CIR}
In this section the methodology developed in
Section~\ref{sec:linearizedFiltering} is applied to the problem of
filtering a Cox-Ingersoll-Ross process. We compare the approximation
via our linearized filtering functional (LFF) (respectively the
induced affine functional filter (AFF)) and other existing approximate
filtering methods to the true solution.
\subsection{Problem formulation}
A Cox-Ingersoll-Ross (CIR) process is a weak solution to the
stochastic differential equation
\begin{equation} \label{eq:CIRdynamics} d X_t = (b + \beta X_t) d t +
  \sigma \sqrt{X_t} d B_t , \quad X_0 = x, \end{equation} where
$b \geq 0$, $\beta \in \R$, $\sigma > 0$ and $B$ is a Brownian
motion. Denoting by $\P_x$ the law of $X$, this gives rise to a
conservative affine process with state space $D = \R_+$. The
parameters in \eqref{eq:collectionOfParameters} are given as
$(0,\sigma^2,b,\beta,0,0,0,0)$.  Let $W$ a Brownian motion independent
of $X$, $\Gamma > 0$ and set
\begin{equation} \label{eq:CIRObservation} Y_t = \int_0^t X_s \dd s +
  \Gamma W_t, \quad t \geq 0.
\end{equation}
The goal is to calculate, for any $t \geq 0$, the distribution of
$X_t$ conditional on the $\sigma$-algebra generated by
$(Y_s)_{s \in [0,t]}$ (see Section~\ref{subsec:filtering}). In
particular, we are interested in the conditional mean and variance
\begin{equation}\label{eq:condMeanAndVar}\begin{aligned}
    \hat{x}_t & = \E_x[X_t | \mathcal{F}_t^Y], \\
    V_t & = \E_x[(X_t-\hat{x}_t)^2 | \mathcal{F}_t^Y], \quad t \geq 0.
  \end{aligned} \end{equation}

There are various methods available to numerically approximate
\eqref{eq:condMeanAndVar}.  For any of these methods one has to pass
to a setup of discrete-time observations at some stage. To do this we
fix $T > 0$, $N \in \mathbb{N}$ and a time-grid
$0 = t_0 < t_1 < \ldots < t_N = T$. Instead of observing the entire
path \eqref{eq:CIRObservation}, one observes at time $t_i$ the random
variable
\begin{equation}\label{eq:discreteObsModel} y_i = X_{t_{i}} (t_{i} -
  t_{i-1}) + \Gamma \sqrt{t_{i} - t_{i-1}} \varepsilon_i
  , \end{equation}
for $i=1,\ldots N$, where $\varepsilon_1,\ldots \varepsilon_N$ are i.i.d. standard normal random variables. This amounts to discretizing the integral in \eqref{eq:CIRObservation} using a Riemann sum and setting $y_i = Y^{disc}_{t_i} - Y^{disc}_{t_{i-1}}$. The filtering distribution is then approximated as  $\E_x[f(X_{t_n}) | \mathcal{F}_{t_n}^Y] \approx \pi_{t_n}^N(f)$ with
\begin{equation} \label{eq:filteringApprox} \pi_{t_n}^N(f) :=
  \E_x[f(X_{t_n}) | \mathcal{F}_{t_n}^{Y,N}], \end{equation} for any
measurable $f \colon \R_+ \to \C$ satisfying
$\E_x[|f(X_t)|] < \infty$, where
$\mathcal{F}_{t_n}^{Y,N} = \sigma(y_1,\ldots y_n)$ and $n=1,\ldots
N$. In particular, instead of \eqref{eq:condMeanAndVar} in what
follows we will denote
\begin{equation}\label{eq:condMeanAndVarApprox}\begin{aligned}
    \hat{x}_t & = \E_x[X_t | \mathcal{F}_t^{Y,N}], \\
    V_t & = \E_x[(X_t-\hat{x}_t)^2 | \mathcal{F}_t^{Y,N}], \quad t \in
    \{t_0,\ldots t_N\}.
  \end{aligned} \end{equation}

\subsection{Numerical solution: Approximate filtering methods}
There are various methods at hand to numerically approximate
\eqref{eq:filteringApprox} and \eqref{eq:condMeanAndVarApprox}. To
illustrate the quality of these we first generate a sample path of the
signal and observation process.  More precisely, a sample of
$(X_{t_0},X_{t_1},\ldots,X_{t_N})$ is generated by (exact) sampling
from the transition density (see \cite[Section~3.4]{Glasserman2004}).
Based on this sample, a sample of $(y_1,\ldots,y_N)$ is generated
using \eqref{eq:discreteObsModel}.

For this sample observation we now compare different methods for
approximating \eqref{eq:filteringApprox} and
\eqref{eq:condMeanAndVarApprox}. As a benchmark we calculate
\eqref{eq:condMeanAndVarApprox} using a (bootstrap) particle filter
with sufficiently many particles ($10^6$ in the examples below), see
\cite[Chapter~10]{Bain2009}. In the plots these results will be
denoted by $\hat{x}$ and $V$ by slight abuse of notation.

This benchmark is now compared to the approximation using the
linearized filtering functional (LFF, developed in the present paper)
and two standard approximations (explained in more detail below): A
Gamma-approximation (\cite{Bates2006}) and a normal approximation
(\cite{Geyer1999}, see also \cite{Brigo1998}). The respective
approximations to \eqref{eq:condMeanAndVarApprox} are denoted as
follows:
\begin{itemize}
\item[Normal] $\hat{x}^{(EKF)}$, $V^{(EKF)}$
\item[Gamma] $\hat{x}^{(G)}$, $V^{(G)}$
\item[LFF] $\hat{x}^{(LFF)}$, $V^{(LFF)}$.
\end{itemize}
 
Firstly, let us explain the approximations from \cite{Bates2006} and
\cite{Geyer1999} in more detail. In both cases basic idea is to
\textit{postulate} that (at each time-step $t_n$) the conditional
distribution in \eqref{eq:filteringApprox} belongs to a certain
two-parameter family of probability distributions (Normal in
\cite{Geyer1999} and Gamma in \cite{Bates2006}). Then (at each
time-step $t_n$) one only needs to approximate
\eqref{eq:condMeanAndVarApprox} and determine the two parameters from
this. In \cite{Geyer1999} the updating procedure for
\eqref{eq:condMeanAndVarApprox} is based on the exact formulas for the
mean and variance of a CIR process and the Kalman filter. This can be
seen as a version of the extended Kalman filter. In \cite{Bates2006}
numerical integration on the level of characteristic functions is used
to update \eqref{eq:condMeanAndVarApprox}. We refer to these articles
for more details.  Both approximations \cite{Bates2006} and
\cite{Geyer1999} can be viewed as special cases of the projection
filter (first introduced in \cite{Brigo1998a}), see \cite{Brigo1998}.

Finally, the unconditional mean and variance are denoted by
$\bar{x}_t:= \E_x[X_t]$ and $v_t:=\E_x[(X_t-\bar{x}_t)^2]$. Since
these correspond to a situation where no observations are available, a
comparison of $(\bar{x},v)$ and $(\hat{x},V)$ shows how much
information the (sample path of the) observation $(y_1,\ldots,y_N)$
contains about $X$. Therefore, these are also shown in the plots
below.

\subsection{Discussion}
We now compare the methods introduced above for two sets of
parameters. For both settings the following choices have been made:
\begin{itemize}
\item instead of a constant $x$, the signal process $X$ is started
  from $X_0=\max(0,Z)$, where $Z \sim \mathcal{N}(x_0,s_0^2)$ is
  independent of $B$ and $W$,
\item the time horizon is $T=1$ and the discretization uses an
  equidistant grid $t_i = i T/N $, $i=0,\ldots N$,
\item $N=1000$, $\sigma = 0.04$, $\beta = -0.2$ and
  $s_0=2 \cdot 10^{-5}$.
\end{itemize}
The remaining parameter values differ for the two settings; they are
indicated in the caption of the figures.

\textbf{Case 1} We choose $b=10^{-6}$, $\Gamma = x_0 = 0.005$.
Figures~\ref{FigAff1} and \ref{FigAff2} show the same sample path of a
CIR process. The sample of observations is not shown in the plot, but
one clearly sees that for $t$ sufficiently large the conditional mean
$\hat{x}$ is neither very close to $X$ nor very close to the mean
$\hat{x}$. Thus, the filtering problem is indeed not trivial: the
posterior distribution in \eqref{eq:filteringApprox} is neither close
to the distribution of $X_{t_n}$ nor concentrated at $X_{t_n}$.

In both figures the conditional mean $\hat{x}$ is shown along with
(dotted) ``confidence bounds'' given by $\hat{x} + \sqrt{V}$ and
$\hat{x} - \sqrt{V}$. This allows to show both conditional mean and
variance in the same plot. The analogous bounds are also shown for the
unconditional mean and the different approximations.

The two figures illustrate that the linearized filtering functional
provides a more accurate approximation for
\eqref{eq:condMeanAndVarApprox} than the standard methods.

\textbf{Case 2} We choose $b=2 \cdot 10^{-5}$,
$\Gamma = x_0 = 0.0001$. In this case both the approximation using the
linearized filtering functional (LFF) and the normal approximation are
not very good. However, it appears that the LFF-approximation becomes
better as $t$ approaches $1$. Although this behaviour is typical in
the present parameter regime, a precise explanation (possibly based on
ergodicity properties of the CIR process) is presently not available.

\begin{figure}
  \centering
  \includegraphics[trim={1cm 1cm 1cm
    1.6cm},clip,width=0.9\textwidth]{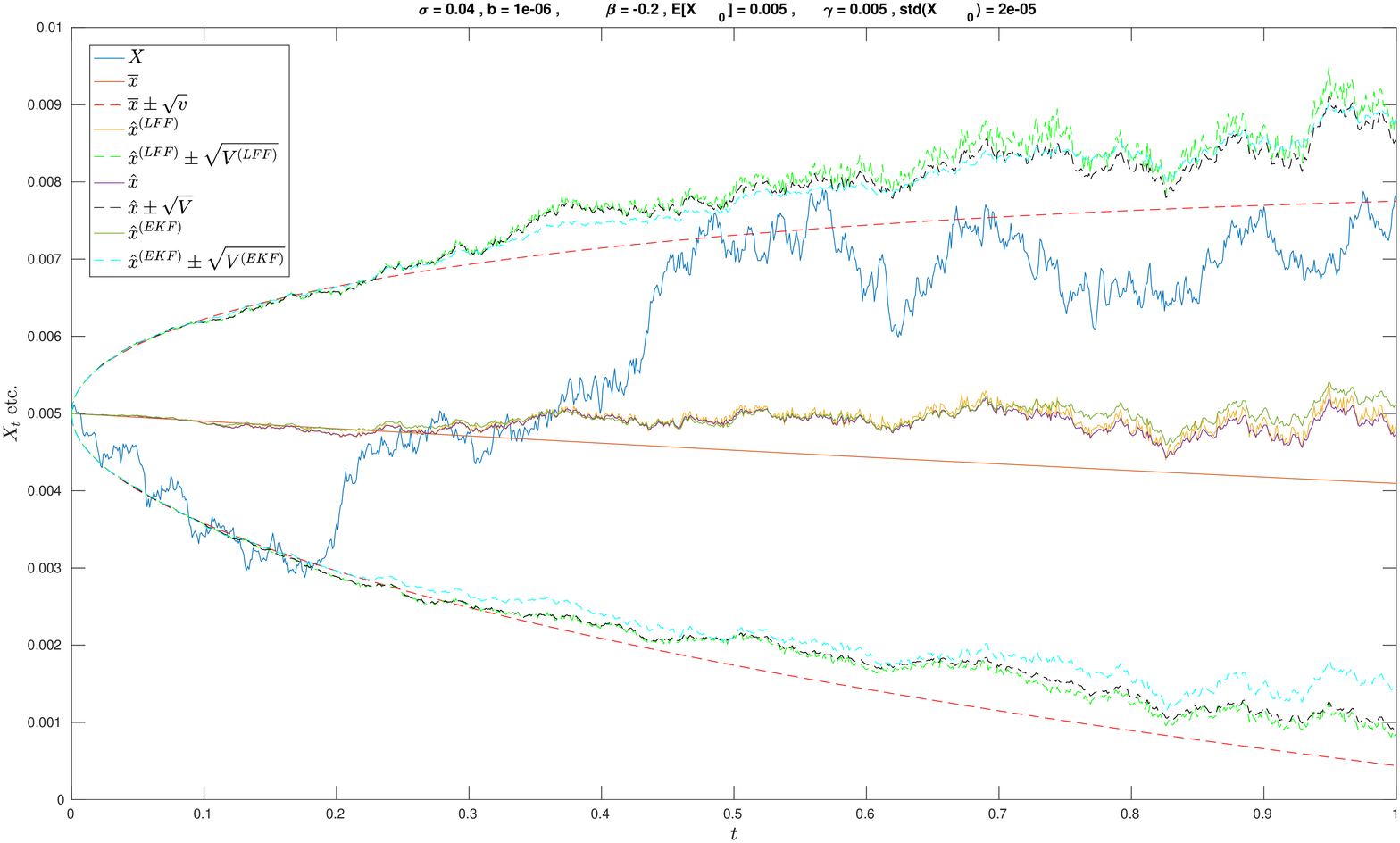}
  \caption{Case 1: Comparison with extended Kalman filter.}
  \label{FigAff1}
\end{figure}

\begin{figure}
  \centering
  \includegraphics[trim={1cm 1cm 1cm
    1.6cm},clip,width=0.9\textwidth]{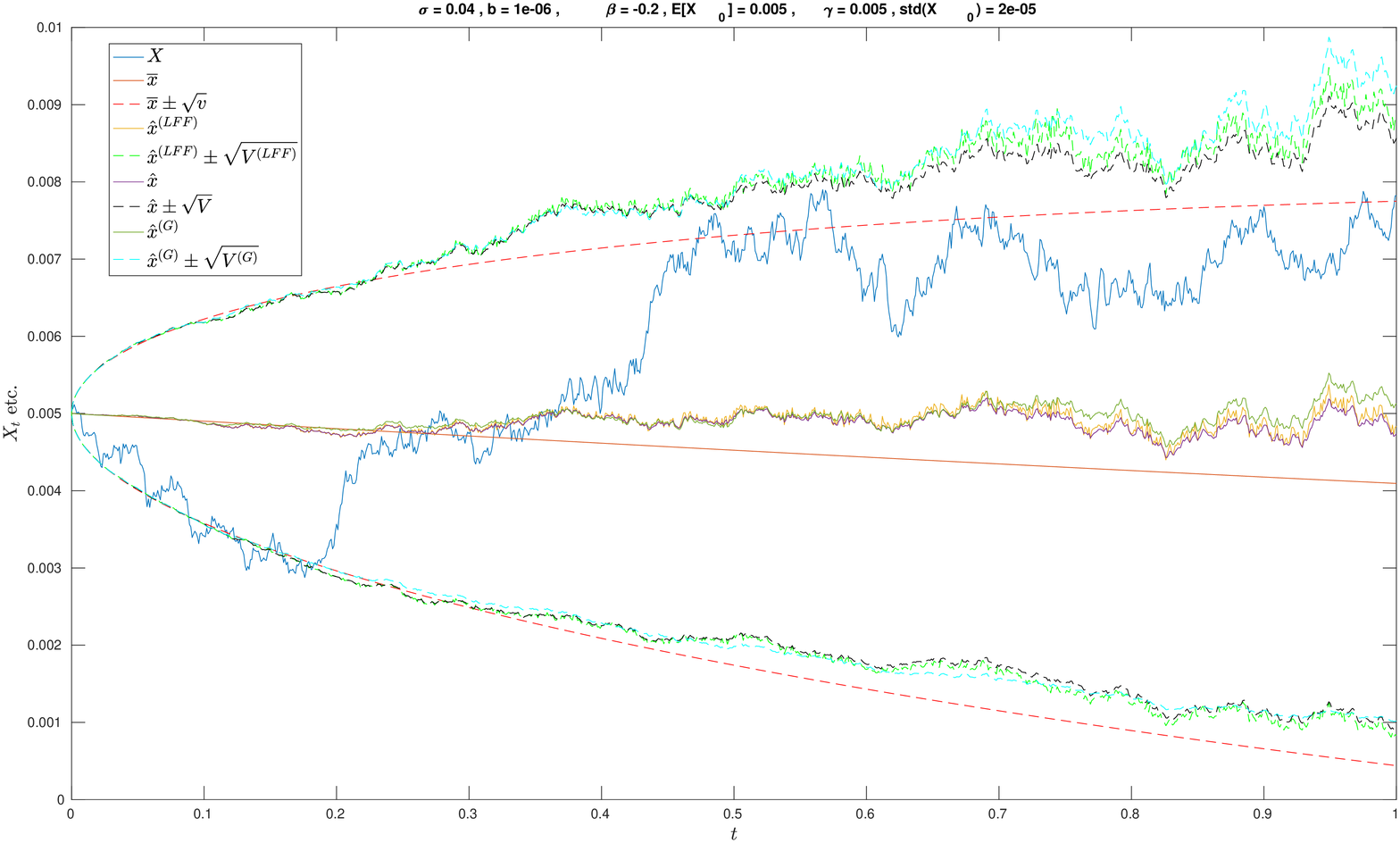}
  \caption{Case 1: Comparison with Gamma approximation.}
  \label{FigAff2}
\end{figure}

\begin{figure}
  \centering
  \includegraphics[trim={1cm 1cm 1cm
    1.4cm},clip,width=0.9\textwidth]{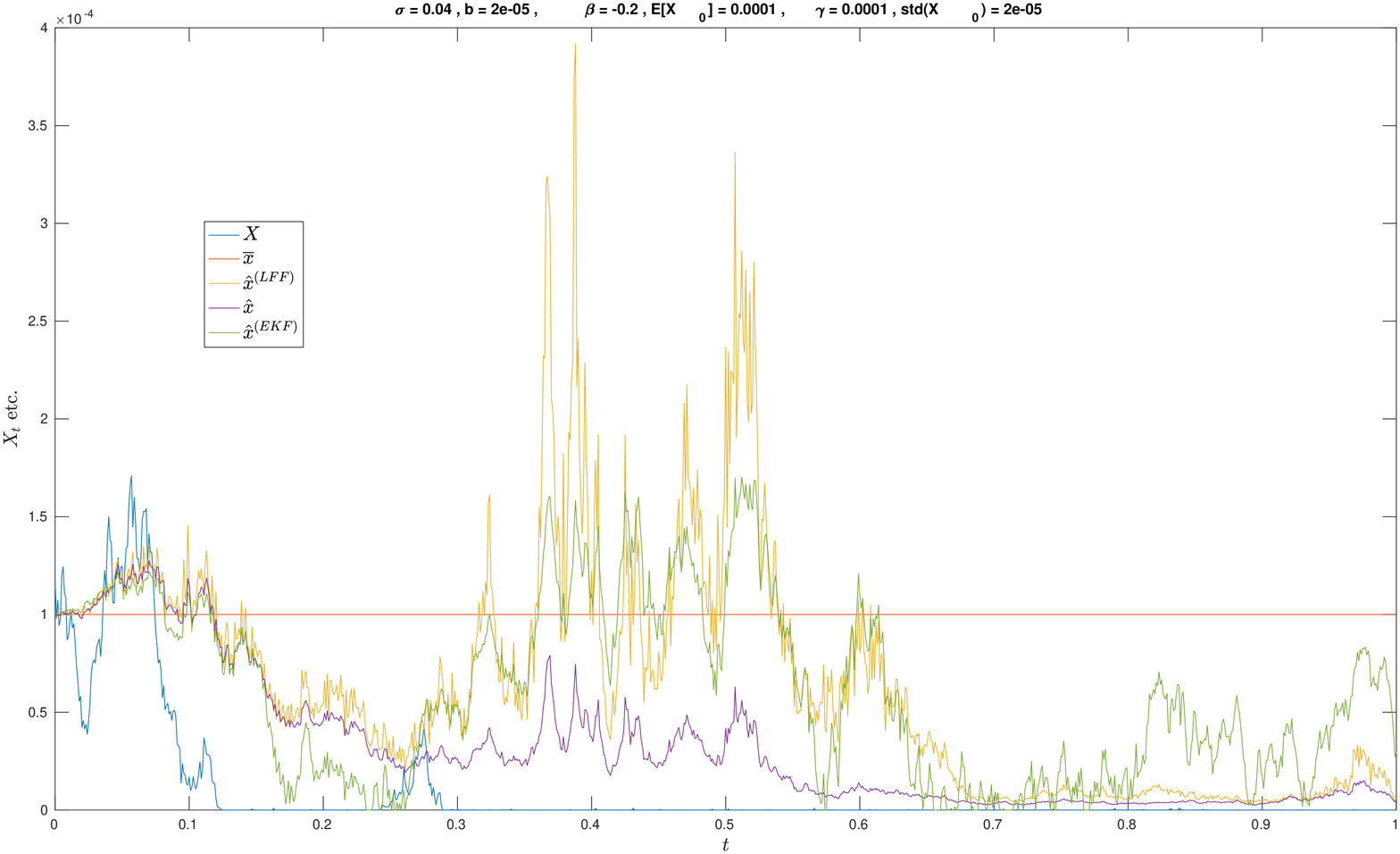}
  \caption{Case 2: Comparison with extended Kalman filter.}
  \label{FigAff3}
\end{figure}

\section{Illustration: Filtering a Wishart process}\label{sec:Wishart}

So far this article has been concerned with the filtering problem for
$\R^m_+ \times \R^{d-m}$-valued affine processes. We now test the
methodology on Wishart processes, an $S^+_d$-valued generalization of
the CIR process (as studied in Section~\ref{sec:CIR}). Here $S^+_d$
denotes the set of all symmetric, positive semidefinite $d \times d$
matrices. Wishart processes were introduced in \cite{Bru1991} and are
commonly used for multivariate stochastic volatility modeling. They
are a subclass of $S^+_d$-valued affine processes as characterized in
\cite{Cuchiero2011}.

Although in theory sequential Monte Carlo methods can be applied for
numerically filtering Wishart processes, in practice this is
infeasible for $d \geq 3$ (see below). Hence, so far no numerical
method has been available for this problem.  We fill this gap by
introducing a linearized filtering functional analogous to
\eqref{eq:modifiedZakaiFilter} and perform numerical experiments for
$d=3$. This section contains simulation results. A generalization of
the theory in Sections~\ref{sec:linearizedFiltering} and
\ref{sec:proofsAff} to $S^+_d$-valued affine processes will be subject
of future work.

\subsection{The signal process}\label{sec:WishartSummary}
Denote by $S^+_d$ the set of all symmetric, positive semidefinite
$d \times d$ matrices and set $S^-_d = - S^+_d$. A Wishart process is
(an $S^+_d$-valued) weak solution to
\begin{equation}\label{eq:WishartSDE} d X_t = (b + H X_t + X_t H^\top)
  \d t + \sqrt{X_t} \d B_t \Sigma + \Sigma^\top \d B_t^\top \sqrt{X_t}
  , \quad X_0 = x, \end{equation}
for $B$ a $d \times d$-matrix of independent standard Brownian motions and suitable $b \in S^+_d$, $x \in S^+_d$, $H \in \R^{d \times d}$, $\Sigma \in \R^{d \times d}$. For simplicity, we assume that $\Sigma \in S^+_d$, $H = 0$, $b = n \Sigma^2$ for some $n \in \mathbb{N}$ with  $n \geq d+1$ and that $x$ has distinct eigenvalues. Then \cite[Proof of Theorem~2'']{Bru1991} ensures that \eqref{eq:WishartSDE} has a unique strong solution for all $t \geq 0$. It also ensures that sample paths of $X$ can be simulated easily: Given $z_0 \in \R^{n \times d}$ with $x = z_0^\top z_0$ and an $n \times d$-Brownian motion $W$, set $Z_t = W_t \Sigma + z_0$ for $t \geq 0$. Then $X:=Z^\top Z$ is a weak solution to \eqref{eq:WishartSDE}. Hence, to simulate a sample path of $X$ one only needs to simulate a sample path of $W$ and apply these two transformations.
Finally, for $u,v \in S_d$ (the set of symmetric $d \times d$-matrices) define $\langle u , v \rangle_{S_d}:=\tr(u v)$. Then for $t \geq 0$ the Laplace transform of $X_t$ is given by
\[ \E[e^{\langle u , X_t \rangle_{S_d}}] = \exp(\phi(t,u)+\langle
  \psi(t,u),x \rangle_{S_d}) , \quad u \in S^-_d \] for some
$\phi:\R_{\geq 0} \times S^-_d \to \R_{-}$ and
$\psi: \R_{\geq 0} \times S^-_d \to S^-_d$. In fact $\phi$ and $\psi$
solve generalized Riccati equations \eqref{eq:genRiccati} with
$R(u):= 2 u \Sigma^2 u$, $F(u):= n \tr(\Sigma^2 u)$.
\subsection{Numerical solution of the filtering problem}
Fix $h\colon S^+_d \to \R^m$ linear and $\Gamma \in \R^{m \times m}$
symmetric, invertible. The observation process $Y$ is defined as
\[ Y_t = \int_0^t h(X_s)\dd s + \Gamma W_t, \quad t \geq 0, \] where
$W$ is an $m$-dimensional Brownian motion independent of $X$, and (the
signal process) $X$ is a solution to \eqref{eq:WishartSDE} with
parameters as specified above (under $\P$).  As before our goal is to
numerically calculate the distribution of $X_t$ conditional on the
$\sigma$-algebra generated by $(Y_s)_{s \in [0,t]}$, for any
$t \geq 0$. For this two methods are used: Firstly a bootstrap
particle filter as in \cite[Chapter~9]{Bain2009} and secondly the
approximate affine filter (AFF) induced by the linearized filtering
functional (LFF). These are defined analogously to the case of a
canonical state space. More precisely, fix $x_0 \in S^+_d$ and for
$t \geq 0$, $y \in C(\R_+,\R^m)$ and $f \in B(S^+_d)$ define the LFF
$\rho_t(\cdot,y)$ by
\[\rho_t(f,y)= \E \left[f(X_t)\exp \left(y_t^\top \Gamma^{-2} h(X_t) -
      \int_0^t y_s^\top \Gamma^{-2} \dd h(X_s)- \int_0^t h(x_0)^\top
      \Gamma^{-2} h(X_s) \dd s \right)\right] \] and the AFF
$\bar{\pi}_t(\cdot,y)$ by \eqref{eq:approximateFilter}. As in
Remark~\ref{rmk:obsNoiseScale} the LFF is obtained by linearizing the
pathwise filtering functional (associated to the observation process
$\Gamma^{-1} Y$ and observation function $\Gamma^{-1} h$) at
$x_0$. Denoting by $h^\top$ the adjoint\footnote{By definition, this
  is the unique linear map $h^\top \colon \R^m \to S_d$ such that
  $ h(x)^\top y = \langle x, h^\top(y)\rangle_{S_d} $ for all
  $y \in \R^m$.} of $h$ and setting
$\bar{y}_s = h^\top(\Gamma^{-2} y_s)$ and
$\bar{x}_0 = h^\top(\Gamma^{-2} h(x_0)) $ one rewrites $\rho_t(f,y)$
as
\[ \rho_t(f,y)= \E \left[f(X_t)\exp \left(\langle \bar{y}_t , X_t
      \rangle_{S_d} - \int_0^t \langle \bar{y}_s , \dd X_s
      \rangle_{S_d} - \int_0^t \langle \bar{x}_0 , X_s \rangle_{S_d}
      \dd s \right)\right] \] and based on
Section~\ref{sec:linearizedFiltering} one expects
\[ \bar{\pi}_t(f,y) = \E_{\mathbb{Q}_{x}^{y,t}}[f(X_t)],\] where under
$\mathbb{Q}_{x}^{y,t}$, $X$ satisfies $X_0 = x$ and
\begin{equation}\label{eq:WishartFilter} d X_s = (n\Sigma^2 + H_s X_s
  + X_s H_s^\top) \d s + \sqrt{X_s} \d B_s \Sigma + \Sigma \d B_s^\top
  \sqrt{X_s}, \quad s \in (0,t]\, ,\end{equation}
with $H_s = 2 \Sigma^2 (\Psi(s) - \bar{y}_s)$, $B$ a $d \times d$ Brownian motion under $\mathbb{Q}_{x}^{y,t}$ and $\Psi$ the solution to  
\begin{equation} \begin{aligned} -\partial_s \Psi(s) &=
    R(\Psi(s)-\bar{y}_s)-\bar{x}_0 , \quad s \in [0,t) \\ \Psi(t) & =
    \bar{y}_t. \end{aligned}
\end{equation}
In particular, \eqref{eq:WishartFilter} yields an ordinary
differential equation for the approximate conditional mean
$\hat{X}_t = \bar{\pi}_t(\mathrm{id},y)$ at time $t$: Formally taking
expectations in \eqref{eq:WishartFilter} one obtains $\hat{X}_0 = x$
and
\begin{equation}\label{eq:WishartMeanApprox} \frac{\d \hat{X}_s}{\d s}
  = (n\Sigma^2 + H_s \hat{X}_s + \hat{X}_s H_s^\top), \quad s \in
  (0,t].  \end{equation}

\subsection{Discussion} We now compare the two methods in an
example. The following choices have been made:
$h(x):=\mathrm{vech}(x)$ is the half-vectorization operator (which
takes the elements of $x$ in the lower triangular part and writes them
in an $m$-dimensional column vector) and $m = \frac{1}{2}
d(d+1)$. Denote by $I_d$ the $d \times d$ identity matrix. We choose
$d=3$, $\Gamma = \Gamma_0 I_3$, $\Sigma = \sigma I_3$ and the
parameter values as shown in the following summary:
\begin{equation}\label{eq:WishartSummary}\begin{aligned} \d X_t & = n \sigma^2 I_3 \d t +  \sigma X_t^{1/2} \d B_t + \sigma \d B_t^\top X_t^{1/2}, \quad X_0 = x_0 \\
    \d Y_t & = \mathrm{vech}(X_t) \d t + \Gamma_0 \d W_t , \quad Y_0 =
    0 \\ (n,\sigma,x_0,\Gamma_0) &=
    (4,0.04,\mathrm{diag}(0.75^2,0.5^2,0.25^2),0.06). \end{aligned} \end{equation}
The filtering problem is discretized analogously to the case of a CIR
process discussed in detail in Section~\ref{sec:CIR}. We choose $T=1$
and equidistant time-points $t_i = i T / N$, $i=0,\ldots N$ with
$N=100$. (Exact) samples of $(X_{t_0}, X_{t_1},\ldots X_{t_N})$ can be
generated as explained in Section~\ref{sec:WishartSummary} and a
spline interpolation is used to generate a continuous observation path
$y$ from discrete measurements.

In this setting the conditional mean $\hat{x}_t$ (see
\eqref{eq:condMeanAndVar} and \eqref{eq:condMeanAndVarApprox}) is
approximated by
\begin{itemize}
\item $\hat{x}^{(PF)}_t$ based on a bootstrap particle filter with
  $N_p$ particles (as in \cite[Chapter~10]{Bain2009}),
\item $\hat{x}^{(AFF)}_t = \hat{X}_t$ in \eqref{eq:WishartMeanApprox}.
\end{itemize}
The computation time required to calculate $\hat{x}^{(PF)}_t$ with a generic implementation on a
standard laptop is enormous already for moderate $N_p$ (e.g. for
$N_p = 10^3, 10^4, 10^5$ it takes roughly $10$ seconds, $1$ minute,
$10$ minutes, respectively). On the other hand, in all these cases the
approximation is very bad and so, in contrast to
Section~\ref{sec:CIR}, here no benchmark is available.  The two
approximations are therefore compared based on their mean square
error: We generate $M$ sample paths of \eqref{eq:WishartSummary},
calculate the approximate conditional mean with both methods and
calculate the average at each time-point,
\[ e^{(m)}_t = \frac{1}{M}\sum_{j=1}^M \|X_{t}^{j} - x^{(m),j}_{t}\|^2
  , \quad t=t_0,\ldots,t_N \] for $m \in \{PF,AFF\}$. Here $X^{j}$ is
the $j$-th sample path of $X$, $x^{(m),j}$ is the approximate
conditional mean (calculated using method $m$) associated to it and
$\|u\|^2:=\langle u,u \rangle_{S_d}$ for $u \in S_d$. By the law of
large numbers and the definition of $\hat{x}$, a smaller value of
$e^{(m)}_{t_i}$ indicates that (on average) $\hat{x}_{t_i}$ and
$x_{t_i}^{(m)}$ are closer.

Figure~\ref{FigAff4} shows a plot of $(t_i,e^{(m)}_{t_i})$,
$i=0,\ldots,N$ for $m \in \{PF,AFF\}$, $N_p=10^4$ and $M=100$. For
this number of particles the calculation of $\hat{x}^{(PF)}$ takes
about $15$ times longer than the calculation of $\hat{x}^{(AFF)}$ (on
average). Nevertheless, the approximation quality of a bootstrap
particle filter is considerably worse than that of the AFF, since the
average mean-square error is significantly larger for longer
time-periods, as shown in Figure~\ref{FigAff4}.

\begin{figure}
  \centering
  \includegraphics[trim={1cm 1cm 1cm
    1.6cm},clip,width=0.9\textwidth]{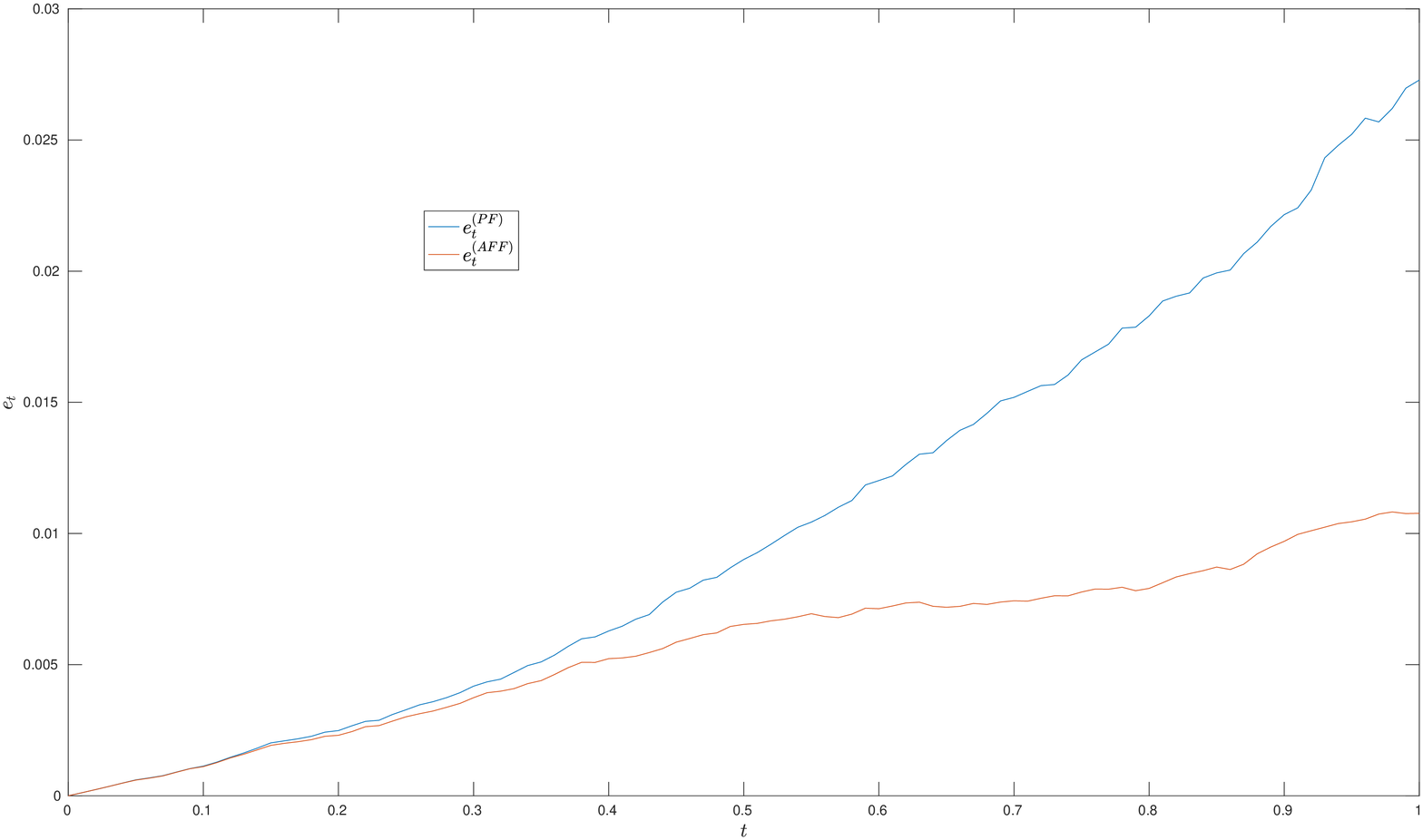}
  \caption{Comparison of mean square error for conditional mean of
    particle filter and AFF.}
  \label{FigAff4}
\end{figure}

 \bibliographystyle{amsalpha}

\end{document}